\documentclass[a4paper,a4wide,10pt]{article} 
\usepackage[utf8]{inputenc}
\usepackage[french, english]{babel}
\usepackage{amsmath,amsthm, amssymb}
\usepackage{lmodern}
\usepackage[usenames,dvipsnames]{pstricks}
\usepackage{enumitem} 
\usepackage{chngcntr}
\usepackage{dsfont}
\usepackage{titlesec}
\usepackage{geometry}
\usepackage{color}
\usepackage{mathtools}
\usepackage{graphicx}
\usepackage{esint}
\usepackage[titletoc,toc,title]{appendix}
\usepackage{mathrsfs}
\usepackage[pdfusetitle=true,bookmarksdepth=3]{hyperref}
\usepackage{amsrefs}
\usepackage{amsfonts}
\usepackage{accents}

\newcommand\keywords[1]{
    \begingroup
   \let\and\\
    \par
    \noindent\textit{Keywords:} #1\par
    \endgroup
}
\newcommand\mathsubj[1]{
    \begingroup
   \let\and\\
    \par
    \noindent\textit{2020 Mathematics Subject Classification:} #1\par
    \endgroup
}

\providecommand{\st}{\; \vert \;}
\newcommand\stSymbol[1][]{%
\nonscript\;#1\vert
\allowbreak
\nonscript\;
\mathopen{}}
\DeclarePairedDelimiterX\set[1]\{\}{%
\renewcommand\st{\stSymbol[\delimsize]}
#1
}

\newtheorem{theorem}{Theorem}[section]
\newtheorem{lemma}[theorem]{Lemma}
\newtheorem{prop}[theorem]{Proposition}

\theoremstyle{definition}
\newtheorem{defn}[theorem]{Definition}
\newtheorem{cor}[theorem]{Corollary}
\newtheorem{rem}[theorem]{Remark}

\DeclareMathOperator*\loc{loc}

\DeclareMathOperator*\supp{supp}
\DeclareMathOperator*\tr{tr}
\DeclareMathOperator*\dom{dom}
\DeclareMathOperator*\var{var}
\newcommand*\diff{\mathop{} d}
\newcommand\numberthis{\addtocounter{equation}{1}\tag{\theequation}}
\newcommand{\mres}{\mathbin{\vrule height 1.7ex depth 0pt width
		0.13ex\vrule height 0.13ex depth 0pt width 1.3ex}}

\numberwithin{equation}{section}
\numberwithin{figure}{section}

\author{Bohdan Bulanyi \footnote{Universit\'e catholique de Louvain, Institut de Recherche en Math\'ematique et Physique, Chemin du Cyclotron 2 bte L7.01.01, 1348 Louvain-la-Neuve, Belgium.  E-mail: bohdan.bulanyi@uclouvain.be}}
\date{}

\begin{document}
	\font\myfont=cmr12 at 15pt
	\title{\myfont On the equivalence of static and dynamic weak optimal transport}
	\maketitle
	\titlelabel{\thetitle.\quad}
\fontsize{10.5}{12}\selectfont
	
	\begin{abstract}
	We show that there is a PDE formulation in terms of Fokker-Planck equations for weak optimal transport problems. The main novelty is that we introduce a minimization problem involving Fokker-Planck equations in the extended sense of measure-valued solutions and prove that it is equal to the associated weak transport problem.  	\end{abstract}
	
\begin{otherlanguage}{french}
\begin{abstract}
Nous montrons qu'il existe une formulation EDP en termes d'équations de Fokker-Planck pour des problèmes de transport optimal faible. La principale nouveauté est que nous introduisons un problème de minimisation dont la contrainte est constituée des équations de Fokker-Planck pour les mesures générales et prouvons qu'il est égal au problème de transport optimal faible correspondant.  
\end{abstract}
\end{otherlanguage}
	
\mathsubj{49Q22; 28A33; 49J45; 49N15; 49Q20.}
	\section{Introduction}
	\subsection{Statement of the problem}
	Given probability measures $\mu$, $\nu$ on $\mathbb{R}^{d}$, we denote by $\Pi(\mu, \nu)$ the set of all probability measures (called transport plans) on $\mathbb{R}^{d}\times \mathbb{R}^{d}$ whose first and second marginals are $\mu$ and $\nu$, respectively. Disintegrating a transport plan $\gamma \in \Pi(\mu,\nu)$ with respect to its first marginal $\mu$, we obtain a $\mu$-almost everywhere uniquely determined family of probability measures $\{\gamma^{x}\}_{x \in \mathbb{R}^{d}}\subset \mathcal{P}(\mathbb{R}^{d})$ such that $\gamma=\gamma^{x}\otimes \mu$. Let $G: \mathbb{R}^{d}\times \mathcal{P}(\mathbb{R}^{d}) \to [0,+\infty]$ be lower semicontinuous in an appropriate sense and $G(x, \cdot)$ be convex on $\mathcal{P}(\mathbb{R}^{d})$ for each $x \in \mathbb{R}^{d}$. The weak transport problem is then defined by
	\begin{equation*}\label{OWTdefinition}
	H(\mu, \nu)=\inf\left\{\int_{\mathbb{R}^{d}} G(x, \gamma^{x}) \diff \mu(x) \st \gamma \in \Pi(\mu, \nu)\right\}.
	\end{equation*}
This problem was first introduced by Gozlan, Roberto, Samson and Tetali \cite{Gozlan_Roberto_Samson_Tetali}, and shortly thereafter by Alibert, Bouchitté and Champion \cite{Alibert-Bouchitte-Champion-2019}. The weak transport problem has been extensively studied in the literature: the results of existence and duality are established in \cite{Alibert-Bouchitte-Champion-2019, MR4029731, Gozlan_Roberto_Samson_Tetali}; the concept of $C$-monotonicity, which is analogous to cyclical monotonicity, was developed in \cite{Backhoff-Beiglbock-Huesmann-Kallblad-2020, MR4029731, Gozlan-Juillet-2020} in order to provide a characterization of optimizers; in \cite{10.1214/19-AIHP1014, Alibert-Bouchitte-Champion-2019, Backhoff-Beiglbock-Huesmann-Kallblad-2020, Beiglbock-Juillet-2021} the weak transport viewpoint is used to investigate martingale optimal transport problems; for applications of weak transport theory, the reader may consult \cite{10.3150/21-BEJ1346}. 

In \cite{Huesmann-Trevisan-2019}, Huesmann and Trevisan introduced the following minimization problem 
\[
f_{\textup{FPE}}(\mu,\nu)=\inf\left\{\int_{0}^{1}\int_{\mathbb{R}^{d}}f(a_{t}(x))\diff \varrho_{t}(x) \diff t \st \partial_{t} \varrho = \tr\left(\frac{1}{2} \nabla^{2} a \varrho \right),\; \varrho_{0}=\mu,\; \varrho_{1}=\nu\right\},
\]
where $\mu$, $\nu$ are probabilities on $\mathbb{R}^{d}$ with finite second moment, the Fokker-Planck equation $\partial_{t} \varrho=\tr(\frac{1}{2} \nabla^{2} a\varrho)$ in $(0,1)\times \mathbb{R}^{d}$ holds in the weak sense and $f$ is $q$-admissible for some $q>1$ (namely, $f(a)$ behaves like $f(a)\sim |a|^{q}$ for a nonnegative symmetric real $d\times d$ matrix $a$, which means that there exist constants $c, C>0$, independent of $a$, such that $c|a|^{q}\leq f(a)\leq C|a|^{q}$; see Sections~2,~3 in \cite{Huesmann-Trevisan-2019}). They proved that
\begin{equation}\label{BB=FPE}
f_{\textup{FPE}}(\mu, \nu)= \inf\left\{\int_{0}^{1} \mathbb{E}\left[f\left(\langle \dot{X} \rangle_{t} \right) \right]\diff t\right\}, 
\end{equation}
where the infimum is taken over all martingales connecting $\mu$ and $\nu$ whose quadratic variation process $\langle X \rangle$ is absolutely continuous with respect to the Lebesgue measure (see \cite[Theorem~3.3]{Huesmann-Trevisan-2019}). It is worth noting the following. Let $G(x,p)=f_{\textup{FPE}}(\delta_{x}, p)$ for each $x \in \mathbb{R}^{d}$ and for each probability measure $p$ on $\mathbb{R}^{d}$ with finite second moment. Then, in view of \cite[Theorem~4.3]{Huesmann-Trevisan-2019}, the functional $G$ is lower semicontinuous in an appropriate sense and $G(x, \cdot)$ is convex for each $x \in \mathbb{R}^{d}$. Taking into account \eqref{BB=FPE} and using a standard measurable selection argument (see, for instance, Section~7.7 in \cite{MR0511544}), we have
\[
f_{\textup{FPE}}(\mu, \nu)=\inf\left\{\int_{\mathbb{R}^{d}} f_{\textup{FPE}}(\delta_{x}, \gamma^{x}) \diff \mu(x) \st \gamma \in \Pi(\mu, \nu)\right\}=H(\mu,\nu).
\]
This provides an equivalent PDE formulation in terms of Fokker-Planck equations for the weak transport problem $H(\mu, \nu)$, where $G(x,p)=f_{\textup{FPE}}(\delta_{x}, p)$ and $f$ is $q$-admissible for some $q>1$ (the reader may also consult the proof of \cite[Theorem~8.2]{10.3150/21-BEJ1346}, where an equivalent SDE formulation for the weak transport problem is established using a measurable selection argument). We emphasize that the $q$-admissibility property of the function $f$ was crucial in the proof of \eqref{BB=FPE}  in \cite{Huesmann-Trevisan-2019}. Namely, the equality \eqref{BB=FPE} was proved in \cite{Huesmann-Trevisan-2019} only for strictly convex functions $f$ behaving like $f(a)\sim |a|^{q}$ for some $q>1$ and for each nonnegative symmetric real $d\times d$ matrix $a$. The purpose of this paper is to provide an equivalent PDE formulation for $H(\mu, \nu)$ in terms of Fokker-Planck equations when the cost function $f$ is not strictly convex and behaves like $f(a) \sim |a|$ (i.e., for example, for the trace or spectral radius of a nonnegative symmetric real $d\times d$ matrix), thereby narrowing the ``gap'' between the classical static formulation for $H(\mu, \nu)$ and the dynamic formulation involving a PDE. 

In particular, we introduce the minimization problem 
	 \begin{equation*}
	 F(\mu,\nu):=\inf\left\{\int_{0}^{1}\int_{\mathbb{R}^{d}}f\left(\frac{\diff \lambda_{t}}{\diff |\lambda_{t}|} \right) \diff |\lambda_{t}|\diff t \st \partial_{t} \varrho=\tr\left(\frac{1}{2} \nabla^{2} \lambda\right), \; \varrho_{0}=\mu,\; \varrho_{1}=\nu\right\}, 
	 \end{equation*}
where $f$ is a nonnegative convex positively 1-homogeneous function and the Fokker-Planck equation $\partial_{t} \varrho=\tr(\frac{1}{2} \nabla^{2} \lambda)$ in $(0,1)\times \mathbb{R}^{d}$ holds in the weak extended sense of measure-valued solutions (see Section~\ref{subsection 2.2}), namely, $\lambda$ does not have to be absolutely continuous with respect to $\varrho$, in contrast to the Fokker-Planck equations $\partial_{t} \varrho=\tr(\frac{1}{2}\nabla^{2} a\varrho)$ in $(0,1)\times \mathbb{R}^{d}$ considered by Huesmann and Trevisan in \cite{Huesmann-Trevisan-2019}, where the analysis is carried out for $q$-admissible costs. By defining $G(x,p)=F(\delta_{x}, p)$, our goal is to prove the equality $H(\mu,\nu)=F(\mu,\nu)$. Let us highlight that the \emph{dynamic} problem $F(\mu,\nu)$ is in fact essentially \emph{static} (see Proposition~\ref{prop: 1.7})  and, unlike the case of $q$-admissible costs \cite{Huesmann-Trevisan-2019}, we do not have an equivalent Benamou-Brenier type formulation \eqref{BB=FPE} for $F(\mu,\nu)$ (see Remark~\ref{rem about abscont}), and hence we cannot use a standard measurable selection argument to prove the equality between $F(\mu, \nu)$ and $H(\mu,\nu)$. Therefore, we implement a new analytical approach. 
\subsection{Main results}
In Proposition~\ref{prop: 1.7}, we obtain a formulation that states that $F(\mu,\nu)$ is in some sense a \emph{static} problem, namely
\[
F(\mu,\nu)=\inf\left\{\int_{\mathbb{R}^{d}}f \left(\frac{\diff \lambda}{\diff |\lambda|}\right) \diff |\lambda| \st \mathrm{tr}\left(\frac{1}{2}\nabla^{2}\lambda\right) = \nu -\mu \right\}.
\]
In Theorem~\ref{prop: 1.13}, we prove the existence of a solution to $F(\mu,\nu)$ whenever $F(\mu,\nu)$ is finite and derive a dual formulation for $F(\mu, \nu)$. This dual formulation is further refined in Theorem~\ref{prop: 1.36} in terms of invariant functions under $G$-transform $\varphi \mapsto \varphi^{G}$ (see \eqref{G-transform ABC}) in line with \cite{Alibert-Bouchitte-Champion-2019} and \cite{Gozlan_Roberto_Samson_Tetali}. In Proposition~\ref{proposition 1.17}, we prove the existence of a solution to $H(\mu,\nu)$ whenever $H(\mu,\nu)$ is finite and deduce a dual formulation for $H(\mu, \nu)$ relying on the $G$-transform, which is then refined in Theorems~\ref{cor: 1.20} and \ref{thmdualHA1-A4}, where the supremum is taken over potentials whose opposite is fixed by the $G$-transform. Since the dual competitors for $H(\mu,\nu)$ are less regular than the dual competitors for $F(\mu,\nu)$, and the supremum is taken for the same functional for both $F(\mu,\nu)$ and $H(\mu,\nu)$ (see Remark~\ref{comment about dformulations for F and H}), we first prove the equality $F(\mu, \nu)=H(\mu,\nu)$ when $G$-invariant functions can be approximated by smooth $G$-invariant functions in a suitable sense (see Theorems~\ref{thmsmooth3},~\ref{thmsmooth4}). Next, under the coercivity assumption on $f$, we characterize bounded continuous $G$-invariant functions as viscosity solutions of the Hamilton-Jacobi-Bellman equation (see Proposition~\ref{prop G-ivariant=VS}). This, under the coercivity and growth assumptions on $f$, yields the dual formulation for $H(\mu,\nu)$, where the competitors are viscosity solutions of the Hamilton-Jacobi-Bellman equation (see Corollary~\ref{corHthroughVS}). Then we perform a smoothing procedure and obtain the equality $F(\mu,\nu)=H(\mu,\nu)$ when the function $f$ behaves like $f(a)\sim |a|$ (see Theorem~\ref{th F=G}). It is worth noting that the established equality between $F(\mu,\nu)$ and $H(\mu,\nu)$ allows us to recover the result of Ghoussoub, Kim and Lin on the existence of a Brownian martingale (see Remark~\ref{rem existence of Brownian martingale}), as well as the Strassen theorem (see Remark~\ref{rem Strassen theorem}). 

 \subsection{Structure of the paper}
In Section~\ref{Section 2}, we introduce our main notation and recall some definitions. In Section~\ref{Section 3}, we announce our basic assumptions on the cost function $f$ and introduce the minimization problem \eqref{1.7} involving Fokker-Planck equations in the extended sense of measure-valued solutions. We prove Proposition~\ref{prop: 1.7} and Theorem~\ref{prop: 1.13}.  We deduce that the functional $F$ is  convex, subadditive and lower semicontinuous in an appropriate sense (see Proposition~\ref{prop: 1.16}). Next, we introduce the weak transport problem~\eqref{thedefofH} (static formulation) associated with the problem~\eqref{1.7} (PDE formulation) and prove Proposition~\ref{proposition 1.17}. We develop the theory of subadditive cost functionals that appeared in \cite[Section~6]{Alibert-Bouchitte-Champion-2019}, where the role of $\mathbb{R}^{d}$ is replaced by the closure of a bounded open convex subset of $\mathbb{R}^{d}$. Unlike \cite{Alibert-Bouchitte-Champion-2019}, this paper considers the set of probability measures on $\mathbb{R}^{d}$ that is not compact, which leads to additional difficulties. We obtain new results either under the coercivity assumption on $f$ (see Propositions~\ref{prop dual HI},~\ref{prop equiv IMS}) or under the growth assumption on $f$ (see Propositions~\ref{prop dual for H through Phi},~\ref{prop equiv PhiMS}). We prove Theorems~\ref{cor: 1.20},~\ref{thmdualHA1-A4}~and~\ref{prop: 1.36}. In Section~\ref{Section 4}, the equality between $F(\mu,\nu)$ and $H(\mu, \nu)$ is proved: under the approximation assumptions on $G$-invariant functions (see Theorems~\ref{thmsmooth3},~\ref{thmsmooth4}); under the coercivity and growth assumptions on $f$ (see Theorem~\ref{th F=G}). We consider in detail four examples \ref{thirdexample}-\ref{levsecondexample} for which we compute $F(\mu,\nu)$ in terms of $\mu$ and $\nu$.

\section{Preliminaries}\label{Section 2}
\subsection{Conventions and Notation}
\textit{Conventions:} in this paper, we say that a value is positive if it is strictly greater than zero, and a value is nonnegative if it is greater than or equal to zero. Euclidean spaces are endowed with the Euclidean inner product $a \cdot b =a^{\mathrm{T}}b$ and the induced norm $|a|=\sqrt{a^{\mathrm{T}}a}$. By $d$ we denote a positive integer. 
\vspace{1.7mm}

\noindent \textit{Notation}:  for $r>0$, $B_{r}(x)$, $\overline{B}_{r}(x)$, and $\partial{B}_{r}(x)$ denote, respectively,  the open ball, the closed ball, and the $(d-1)$-sphere with center $x$ and radius $r$. Let $S_{d}$ be the space of real $d\times d$ symmetric matrices endowed with the Hilbert–Schmidt (or Frobenius) inner product $A:B=\tr(A^{\mathrm{T}}B)$ and the induced norm $\|A\|=\sqrt{A:A}$. We write $I_{d}$ for the $d\times d$ identity matrix and $S^{+}_{d}$ ($S^{++}_{d}$) for the set of symmetric nonnegative (positive) definite real $d\times d$ matrices. We write  $\mathcal{P}_{2}((0,1)\times \mathbb{R}^{d})$ ($\mathcal{P}_{2}(\mathbb{R}^{d})$) for the set of probability measures $\mu$ on $(0,1)\times \mathbb{R}^{d}$ ($\mathbb{R}^{d}$) such that $\int_{(0,1)\times \mathbb{R}^{d}}|x|^{2} \diff \mu(t, x)$ ($\int_{\mathbb{R}^{d}}|x|^{2}\diff \mu(x)$) is finite. We endow $\mathcal{P}_{2}(\mathbb{R}^{d})$ with the topology generated by the $2$-Wasserstein distance (see Definition~6.8 and Theorem~6.9 in \cite{Villani-2009}). We write $\mathcal{M}((0,1)\times \mathbb{R}^{d}, S^{+}_{d})$ ($\mathcal{M}(\mathbb{R}^{d}, S^{+}_{d})$) for the set of measures $\lambda$ on $(0,1)\times \mathbb{R}^{d}$ ($\mathbb{R}^{d}$) with values in $S_{d}$ such that for each Borel set $E\subset (0,1) \times \mathbb{R}^{d}$ ($E\subset \mathbb{R}^{d}$), $ \lambda (E)  \in S^{+}_{d}$.  For a measure $\lambda$, we denote by $|\lambda|$ its total variation. By $C_{b}(\mathbb{R}^{d})$ we denote the space of all real-valued bounded continuous functions on $\mathbb{R}^{d}$. We write $\mathcal{S}_{b}(\mathbb{R}^{d})$ ($\mathcal{U}_{b}(\mathbb{R}^{d})$) for the set of all bounded lower (upper) semicontinuous functions on $\mathbb{R}^{d}$. If $U\subset \mathbb{R}^{d}$ is Lebesgue measurable, then $L^{1}(U)$ will denote the space consisting of all real measurable functions on $U$ that are integrable on $U$. By $L^{1}_{loc}(U)$ we denote the space of functions $u$ such that $u\in L^{1}(V)$ for all $V\Subset U$. We use the standard notation for Sobolev spaces. For an open set $U\subset \mathbb{R}^{d}$, denote by $W^{1,p}_{0}(U)$ the closure of $C^{\infty}_{c}(U)$ in the Sobolev space $W^{1,p}(U)$, where $C^{\infty}_{c}(U)$ is the space of functions in $C^{\infty}(U)$ with compact support in $U$. By $\mathcal{L}^{r}$ and $\mathcal{H}^{r}$ we denote the $r$-dimensional Lebesgue and Hausdorff measure, respectively. For each $\mu \in \mathcal{P}_{2}(\mathbb{R}^{d})$, $[\mu]$ and $\var(\mu)$ will denote the barycenter  and the variance of $\mu$, respectively. Namely, $[\mu]=\int_{\mathbb{R}^{d}}x\diff \mu(x)$ and $\var(\mu)=\int_{\mathbb{R}^{d}}|x|^{2}\diff \mu(x)-|[\mu]|^{2}$.

\subsection{Fokker-Planck equation for general measures}\label{subsection 2.2}

Let $C^{1,2}_{b}((0,1)\times \mathbb{R}^{d})$ be the space of functions continuously differentiable once in $t$ and twice in $x$ in $(0,1)\times \mathbb{R}^{d}$ with uniformly bounded $\partial_{t} \varphi$ and $\nabla^{2}_{x} \varphi$. 

Given a pair of measures $(\varrho, \lambda) \in \mathcal{P}_{2}((0,1)\times \mathbb{R}^{d})\times \mathcal{M}((0,1)\times \mathbb{R}^{d}, S^{+}_{d})$ such that one can disintegrate $\varrho=\varrho_{t} \otimes \mathcal{L}^{1}\mres (0,1)$, $\lambda =\lambda_{t} \otimes \mathcal{L}^{1} \mres (0,1)$, we say that the Fokker-Planck equation in the weak extended sense of measure-valued solutions 
\begin{equation} \label{1.1}
\partial_{t}\varrho= \tr \left(\frac{1}{2} \nabla^{2} \lambda\right) \,\  \text{in}\,\ (0,1) \times \mathbb{R}^{d} \tag{\textup{GFPE}}
\end{equation}
holds if 
\begin{equation*}\label{boundoncompsets}
\int_{0}^{1}|\lambda_{t}|(\mathbb{R}^{d})\diff t < +\infty 
\end{equation*}
and for each $\varphi \in C^{1,2}_{b}((0,1)\times \mathbb{R}^{d})$ with (closed) support in $(0,1) \times \mathbb{R}^{d}$,
\begin{equation}
\int^{1}_{0}\int_{\mathbb{R}^{d}}\partial_{t}\varphi(t,x)\diff\varrho_{t}(x) \diff t= -\int^{1}_{0}\int_{\mathbb{R}^{d}}\frac{1}{2}\nabla^{2}_{x}\varphi(t,x) : d\lambda_{t}(x)\diff t. \label{1.3}
\end{equation}
If $(\varrho, \lambda) \in \mathcal{P}_{2}((0,1)\times \mathbb{R}^{d}) \times \mathcal{M}((0,1)\times \mathbb{R}^{d}, S^{+}_{d})$ is a solution to \eqref{1.1}, then there exists a narrowly continuous curve $(\widetilde{\varrho}_t)_{t\in [0,1]}\subset \mathcal{P}_{2}(\mathbb{R}^{d})$ such that $\widetilde{\varrho}_{t}=\varrho_{t}$ for $\mathcal{L}^{1}$-a.e. $t \in (0,1)$. Moreover,  if $\psi \in C^{2}_{b}(\mathbb{R}^{d})$ and $0\leq t_{1} \leq t_{2}\leq 1$, then
\begin{equation} \label{1.4}
 \int_{\mathbb{R}^{d}}\psi(x) \diff \widetilde{\varrho}_{t_{2}}( x)-\int_{\mathbb{R}^{d}}\psi(x)\diff \widetilde{\varrho}_{t_{1}}(x) = \int_{t_{1}}^{t_{2}}\int_{\mathbb{R}^{d}}\frac{1}{2}\nabla^{2} \psi(x) : d \lambda_{t}(x) \diff t
\end{equation}
(the reader may consult \cite[Lemma~8.1.2]{Ambrosio-Gigli-Savare-2008} and \cite[Remark~2.3]{Trevisan-2016}). Thus, for a solution $(\varrho, \lambda)$ to \eqref{1.1}, without loss of generality, we shall assume that $(\varrho_{t})_{t \in (0,1)}$ is narrowly continuous. 

Let $C^{2}_{b}(\mathbb{R}^{d})$ be the space of functions twice continuously differentiable on $\mathbb{R}^{d}$ whose Hessian is uniformly bounded. 
\begin{prop} \label{prop: 1.4} Let $(\varrho, \lambda)$ solve \eqref{1.1}, $\mu=w*-\lim_{t\searrow0}\varrho_{t}$ and $\nu=w*-\lim_{t\nearrow 1}\varrho_{t}$. If $\psi \in C^{2}_{b}(\mathbb{R}^{d})$ is convex, the function $t \in [0,1] \mapsto \int_{\mathbb{R}^{d}}\psi \diff\varrho_{t}$ is nondecreasing. In particular, $\int_{\mathbb{R}^{d}}\psi \diff \mu \leq \int_{\mathbb{R}^{d}}\psi \diff \nu$.
\end{prop}
\begin{proof} For every $0\leq  t_{1} \leq t_{2} \leq 1$,  using (\ref{1.4}) and the fact that $\psi$ is convex and $\lambda_{t} \in \mathcal{M}(\mathbb{R}^{d}, S^{+}_{d})$, we have
\begin{align*}
\int_{\mathbb{R}^{d}}\psi(x) \diff \varrho_{t_{2}}(x) & = \int_{\mathbb{R}^{d}}\psi(x)\diff \varrho_{t_{1}}(x) + \int^{t_{2}}_{t_{1}}\int_{\mathbb{R}^{d}}\frac{1}{2}\nabla^{2}\psi(x):d \lambda_{t}(x) \diff t \geq \int_{\mathbb{R}^{d}}\psi(x)\diff \varrho_{t_{1}}(x),
\end{align*}
which completes our proof of Proposition~\ref{prop: 1.4}.
\end{proof}
For convenience, we recall the next definition (see \cite{Strassen-1965}).
\begin{defn}\label{defnconvorder} We say that measures $\mu, \nu \in \mathcal{P}_{2}(\mathbb{R}^{d})$ are in convex order, and we write $\mu \le_{c} \nu$, if for each convex function $\psi:\mathbb{R}^{d}\to \mathbb{R}$, it holds 
\[
\int_{\mathbb{R}^{d}}\psi(x) \diff \mu(x) \leq \int_{\mathbb{R}^{d}}\psi(x) \diff \nu(x).
\]
\end{defn}
\begin{rem}\label{remarkaboutconvexorder} Notice that $\mu \le_{c} \nu$ if and only if $\langle \nu-\mu, \varphi \rangle \geq 0$ for each convex function $\varphi \in C^{2}_{b}(\mathbb{R}^{d})$. Indeed, assume that $\langle \nu-\mu, \varphi \rangle \geq 0$ for each convex function $\varphi \in C^{2}_{b}(\mathbb{R}^{d})$. Then $\int_{\mathbb{R}^{d}}\psi \diff \mu = \int_{\mathbb{R}^{d}}\psi \diff \nu$ for each affine map $\psi:\mathbb{R}^{d} \to \mathbb{R}$. Since every convex function is nonnegative up to the addition of some affine map, the relation $\mu \leq_{c} \nu$ will be justified as soon as we show that $\langle \nu-\mu, \varphi \rangle \geq 0$ for each convex function $\varphi: \mathbb{R}^{d} \to [0,+\infty)$. For such a function $\varphi$, there exists a sequence $(\varphi_{k})_{k\in \mathbb{N}}$ of convex functions $\varphi_{k}: \mathbb{R}^{d} \to [0,+\infty)$ such that $\varphi_{k}(\cdot)=\inf\{\varphi(y)+k|\cdot-y| \st y \in \mathbb{R}^{d}\}$ is $k$-Lipschitz and $\varphi_{k} \nearrow \varphi$ as $k \to +\infty$. For each $k \in \mathbb{N}$, using convolution with $\eta_{\varepsilon}(\cdot)=\varepsilon^{-d}\eta(\cdot/\varepsilon)$, where $\varepsilon>0$ and $\eta$ is a standard mollifier (as in Lemma~\ref{lem 1.36}), $\varphi_{k}$ can be uniformly approximated on $\mathbb{R}^{d}$ by a sequence of nonnegative convex smooth $k$-Lipschitz functions. Again, using the above mollification procedure,  each convex smooth Lipschitz function on $\mathbb{R}^{d}$ can be uniformly approximated by a sequence of convex functions lying in $C^{2}_{b}(\mathbb{R}^{d})$. After all, taking into account the above monotone and uniform convergences, for each convex function $\varphi:\mathbb{R}^{d}\to [0,+\infty)$ we can find a sequence $(\varphi_{k})_{k\in \mathbb{N}} \subset C^{2}_{b}(\mathbb{R}^{d})$ of nonnegative convex functions such that $\int_{\mathbb{R}^{d}} \varphi_{k} \diff \mu \to \int_{\mathbb{R}^{d}} \varphi \diff \mu$ and $\int_{\mathbb{R}^{d}} \varphi_{k} \diff \nu \to \int_{\mathbb{R}^{d}} \varphi \diff \nu$ as $k \to +\infty$. This actually justifies the above criterion.
\end{rem}
\subsection{Subharmonic functions}\label{shfssct}
\begin{defn}\label{defofsubharmonic}
A function $u: \mathbb{R}^{d} \to \mathbb{R}\cup\{-\infty\}$ is said to be subharmonic if it satisfies the following conditions.
\begin{enumerate}[label=(\roman*)]
\item $u$ is not identically equal to $-\infty$.
\item  $u$ is upper semicontinuous.
\item \label{assert3subharm} For each $x\in \mathbb{R}^{d}$ and $r>0$, 
\begin{equation*}
u(x)\leq \fint_{\partial B_{r}(x)}u(y)\diff \mathcal{H}^{d-1}(y),
\end{equation*}
where $\fint_{\partial B_{r}(x)}u(y)\diff \mathcal{H}^{d-1}(y)=\frac{1}{\mathcal{H}^{d-1}(\partial B_{r}(x))}\int_{\partial B_{r}(x)}u(y)\diff \mathcal{H}^{d-1}(y)$.
\end{enumerate}
\end{defn}
A function $v: \mathbb{R}^{d}\to \mathbb{R}\cup \{+\infty\}$ is said to be superharmonic if $u=-v$ is subharmonic. A function is harmonic if and only if it is both subharmonic and superharmonic. 
If $u \in C^{2}(\mathbb{R}^{d})$, then $u$ is subharmonic if and only if $\Delta u\geq 0$ in $\mathbb{R}^{d}$ (see, for instance, \cite[Section~3.2]{Rado-1937}).
In one dimension, a function is subharmonic if and only if it is convex; however,  in dimension $d\geq 2$ the notions of subharmonicity and convexity are not equivalent (for more details, the reader may consult, for instance, \cite{Rado-1937}).  We shall denote by $\mathcal{SH}(\mathbb{R}^{d})$ the cone of all subharmonic functions on $\mathbb{R}^{d}$. 

Next, we recall the notion of the so-called \textit{subharmonic order}, which is stronger than the convex order in dimension~$d\geq2$ and is equivalent to the convex order in one dimension. It is worth mentioning that the convex order between a pair of probability measures $\mu$ and $\nu$ is a necessary and sufficient condition for the existence of a martingale coupling between $\mu$ and $\nu$, as was proved by Strassen in \cite{Strassen-1965}. The subharmonic order between $\mu$ and $\nu$ is a necessary and sufficient condition for the existence of a Brownian martingale coupling between $\mu$ and $\nu$ (namely, a transport plan $\gamma \in \Pi(\mu,\nu)$ which is a joint distribution of $(B_{0}, B_{\tau})$, where $(B_{t})_{t}$ is the Brownian motion with initial law $\mu$, the law of $B_{\tau}$ is $\nu$ and $\tau$ is a possibly randomized stopping time for the Brownian filtration), as was proved by Ghoussoub, Kim, and Lin in \cite{Ghoussoub-Kim-Lin}.
	\begin{defn}\label{subharmonic order measures}
		We say that measures $\mu, \nu \in \mathcal{P}_{2}(\mathbb{R}^{d})$ are in subharmonic order, and we write $\mu\leq_{sh} \nu$, if for each $u \in C^{2}_{b}(\mathbb{R}^{d}) \cap \mathcal{SH}(\mathbb{R}^{d})$ the following holds 
		\[
		\int_{\mathbb{R}^{d}}u\diff \mu \le \int_{\mathbb{R}^{d}}u \diff \nu.
		\]
	\end{defn}
If $\mu \leq_{sh} \nu$, using a standard mollification procedure, we observe that $\int_{\mathbb{R}^{d}} u \diff \mu \leq \int_{\mathbb{R}^{d}} u \diff \nu$ for each bounded or Lipschitz $u \in \mathcal{SH}(\mathbb{R}^{d})$.
\subsection{Fenchel conjugate}
Let $X$ be a normed vector space and $X^{\prime}$ be the topological dual space of $X$. 
\begin{defn} Let $f:X\to \mathbb{R}\cup \{+\infty\}$ and $\dom(f)=\{x\in X \st f(x)<+\infty\}\not = \emptyset$. The Fenchel conjugate $f^{*}:X^{\prime}\to \mathbb{R}\cup \{+\infty\}$ of $f$ at $x^{\prime} \in X^{\prime}$ is defined by
\[
f^{*}(x^{\prime})=\sup\{x^{\prime}(x)-f(x) \st x \in X\}.
\]
\end{defn}
If $X=S_{d}$, then the following result holds. Recall that a function $f:S_{d} \to [0, +\infty]$ is said to be positively 1-homogeneous if 
\[ f(tA)=tf(A) \,\  \text{for all}\,\ A, B \in S_{d}\,\ \text{and}\,\ t\geq 0.\]
We denote by $C_{b}(\mathbb{R}^{d}, S_{d})$ the space of all bounded continuous functions on $\mathbb{R}^{d}$ with values into $S_{d}$. 
\begin{lemma}\label{lemma conjugate prop}
Let $f: S_{d} \to [0, +\infty]$ be convex, lower semicontinuous and positively 1-homogeneous with $\dom(f) \subset S^{+}_{d}$ and $\dom(f)\not = \emptyset$. Let $\Psi: (C_{b}(\mathbb{R}^{d}, S_{d}))^{\prime} \to [0,+\infty]$ be defined by
\[
\Psi(\sigma)=\begin{cases}
\displaystyle \int_{\mathbb{R}^{d}}f\left(\frac{\diff \sigma}{\diff |\sigma|}\right) \diff |\sigma| \,\ & \text{if} \,\  \sigma \in \mathcal{M}(\mathbb{R}^{d}, S^{+}_{d}),\,\ |\sigma|(\mathbb{R}^{d})<+\infty,\\
+\infty \,\ & \text{otherwise}.
\end{cases}
\]
Then $\Psi$ is convex and positively 1-homogeneous. For each $\xi \in C_{b}(\mathbb{R}^{d}, S_{d})$, 
\begin{equation} \label{eq computofpsistarconj1}
\begin{split}
\Psi^{*}(\xi)=\begin{cases}
0 \,\ & \text{if} \,\ \xi(\mathbb{R}^{d}) \subset \dom(f^{*}), \\
+\infty \,\ & \text{otherwise}.
\end{cases}
\end{split}
\end{equation}
Furthermore, for each $\sigma \in \mathcal{M}(\mathbb{R}^{d}, S^{+}_{d})$ such that $|\sigma|(\mathbb{R}^{d})<+\infty$, $\Psi$ is (weakly) lower semicontinuous at $\sigma$ and   $\Psi^{**}(\sigma)=\Psi(\sigma)$.
\end{lemma}
\begin{proof}
Given the definition of $\Psi$, to prove the convexity of $\Psi$, it suffices to show that if $t\in (0,1)$, $\sigma_{1}, \sigma_{2}  \in \mathcal{M}(\mathbb{R}^{d}, S^{+}_{d})$ and $|\sigma_{i}|(\mathbb{R}^{d})<+\infty$ for each $i \in \{1,2\}$, then $\Psi(t \sigma_{1}+(1-t)\sigma_{2})\leq t\Psi(\sigma_{1})+(1-t)\Psi(\sigma_{2})$. This estimate comes by using the facts that $f$ is positively 1-homogeneous and convex. Indeed, we have
\begin{equation*}
\begin{split}
\Psi(t \sigma_{1}+(1-t)\sigma_{2})&=\int_{\mathbb{R}^{d}}f\left(\frac{\diff (t\sigma_{1}+(1-t)\sigma_{2})}{\diff |t\sigma_{1}+(1-t)\sigma_{2}|}\right) \diff |t\sigma_{1}+(1-t)\sigma_{2}| \\ & = \int_{\mathbb{R}^{d}}f\left(\frac{\diff (t\sigma_{1}+(1-t)\sigma_{2})}{\diff (t|\sigma_{1}|+(1-t)|\sigma_{2}|)}\right) \diff (t|\sigma_{1}|+(1-t)|\sigma_{2}|)\\
& \leq  t\int_{\mathbb{R}^{d}}f\left(\frac{\diff \sigma_{1}}{\diff (t|\sigma_{1}|+(1-t)|\sigma_{2}|)}\right) \diff (t|\sigma_{1}|+(1-t)|\sigma_{2}|) \\ & \,\ \,\ \,\  +(1-t)\int_{\mathbb{R}^{d}}f\left(\frac{\diff \sigma_{2}}{\diff (t|\sigma_{1}|+(1-t)|\sigma_{2}|)}\right) \diff (t|\sigma_{1}|+(1-t)|\sigma_{2}|)\\
& = t\int_{\mathbb{R}^{d}}f\left(\frac{\diff \sigma_{1}}{\diff |\sigma_{1}|}\right) \diff |\sigma_{1}|+(1-t)\int_{\mathbb{R}^{d}}f\left(\frac{\diff \sigma_{2}}{\diff |\sigma_{2}|}\right) \diff |\sigma_{2}| = t\Psi(\sigma_{1})+(1-t)\Psi(\sigma_{2})
\end{split}
\end{equation*}
(see Remark~\ref{rem: 1.8}). The positive 1-homogeneity of $\Psi$ comes from its definition and the positive 1-homogeneity of $f$.

Next, fix an arbitrary $\xi \in C_{b}(\mathbb{R}^{d}, S_{d})$. If $\xi(x_{0}) \not \in \dom(f^{*})$, then there exists $M \in \dom(f)$ such that  $\xi(x_{0}):M>f(M)$ (see \eqref{mf f usual}). Since $\xi$ is continuous, there exists $r>0$ such that $\xi(x):M>f(M)$ for each $x \in B_{r}(x_{0})$. Defining $\sigma_{n}=n M  \mathcal{L}^{d}\mres B_{r}(x_{0})$ and using the positive 1-homogeneity of $f$, we have 
\[
\Psi^{*}(\xi)\geq \int_{\mathbb{R}^{d}}\xi(x) :\diff \sigma_{n}(x) - \int_{\mathbb{R}^{d}}f\left(\frac{\diff \sigma_{n}}{\diff |\sigma_{n}|}\right) \diff |\sigma_{n}| = n \int_{B_{r}(x_{0})} (\xi(x):M-f(M)) \diff x >0.
\]
Letting $n$ tend to $+\infty$, we deduce that $\Psi^{*}(\xi)=+\infty$. On the other hand, if $\xi(\mathbb{R}^{d})\subset \dom(f^{*})$, then for each $\sigma \in \mathcal{M}(\mathbb{R}^{d}, S_{d}^{+})$ such that $|\sigma|(\mathbb{R}^{d})<+\infty$, we have 
\[
\int_{\mathbb{R}^{d}}\xi(x):\diff \sigma(x)-\int_{\mathbb{R}^{d}}f\left(\frac{\diff \sigma}{\diff |\sigma|}\right) \diff |\sigma| \leq 0
\]
(see \eqref{mf f usual}) and hence $\Psi^{*}(\xi)=0$. 
Thus, we have proved \eqref{eq computofpsistarconj1}. 

Inasmuch as $\Psi$ is convex and for each $\sigma \in \mathcal{M}(\mathbb{R}^{d}, S^{+}_{d})$ such that $|\sigma|(\mathbb{R}^{d})<+\infty$, $\Psi$ is (weakly) lower semicontinuous at $\sigma$ (which is a consequence of \cite[Theorem~2.34]{APD} and Remark~\ref{rem: 1.8}), we have $\Psi^{**}(\sigma)=\Psi(\sigma)$ (see \cite[Theorem~2.1 $(i)$]{BOUCHITTE2006642}). This completes our proof of Lemma~\ref{lemma conjugate prop}.
\end{proof}
\section{A PDE constrained optimization problem}\label{Section 3}
For each $\mu, \nu \in \mathcal{P}_{2}(\mathbb{R}^{d})$, we consider the following minimization problem 
\begin{equation} \label{1.7}
\begin{split}
 F(\mu,\nu)= \inf \left\{ \int^{1}_{0}\int_{\mathbb{R}^{d}}f\left(\frac{\diff \lambda_{t}}{\diff |\lambda_{t}|}\right) \,d |\lambda_{t}| \diff t \st \partial_{t} \varrho=\tr\left(\frac{1}{2} \nabla^{2} \lambda\right), \; \varrho_{0}=\mu,\; \varrho_{1}=\nu \right\}, 
 \end{split}
\end{equation}
where the function $f: S_{d} \to [0,+\infty]$ with $\dom(f):=\{M \in S_{d} \st f(M)<+\infty\} \subset S^{+}_{d}$ satisfies the following assumptions.
\begin{align}
& \,\  \,\ \text{$\dom(f) \cap S^{++}_{d}$ is dense in $\dom(f)\not = \emptyset$}. \tag{$A0$} \label{assumpdom}\\
& \,\  \,\ \text{$f$ is convex and positively 1-homogeneous}. \tag{$A1$} \label{assumpsublinearf}\\
& \,\  \,\ \text{$f$ is lower semicontinuous}. \tag{$A2$}   \label{assumplscf} 
\end{align}
It is worth noting that $f$ is convex and positively 1-homogeneous if and only if $f$ is subadditive and positively 1-homogeneous. Among the interesting examples of such functions, we emphasize the following.
\begin{enumerate}[label=(\roman*)]
\item \label{thirdexample}  $f(A)=t$ if $A=t I_{d}$ for some $t \geq 0$ and $f(A)=+\infty$ otherwise.
\item \label{notcoerc}For some $B \in S^{+}_{d}$,  $f(A)=A: B$ if $A \in S^{+}_{d}$ and $f(A)=+\infty$ otherwise. 
\item \label{tracefirstexample} $f(A)= \tr(A)$ if $A \in S^{+}_{d}$ and $f(A)=+\infty$ otherwise.
\item \label{levsecondexample} $f(A)$ is the largest eigenvalue of $A$ if $A\in S^{+}_{d}$ and $f(A)=+\infty$ otherwise.
\end{enumerate}
Notice that in the example~\ref{thirdexample} $\dom(f)$ is a proper convex subset of $S^{+}_{d}$, in the example~\ref{notcoerc} the function $f$ is not coercive, and the example~\ref{tracefirstexample} is the particular case of the example~\ref{notcoerc} with $B=I_{d}$.

Using the 1-homogeneity of $f$, we compute its Fenchel conjugate. For each $A\in S_{d}$,
\begin{equation}\label{mf f usual}
\begin{split}
f^{*}(A)&=\sup\left\{A:M-f(M) \st  M\in S_{d}\right\}
=\begin{cases}  0 \,\ &\text{if}\,\ \,\ A:M\leq f(M) \,\ \forall M\in \dom(f),\\
+\infty \,\ &\text{otherwise}.
\end{cases}
\end{split}
\end{equation}
In particular, $A \in \dom(f^{*})$ for each $A \in -S^{+}_{d}$.
\begin{rem} \label{rem: 1.6} If $F(\mu,\nu)<+\infty$, then $\mu \le_{c} \nu$ (see Proposition~\ref{prop: 1.4} and Remark~\ref{remarkaboutconvexorder}), which implies that $[\mu]=[\nu]$.
\end{rem}
Using the 1-homogeneity of $f$, we eliminate the time variable in (\ref{1.7}) via the equation
\begin{equation}
\tr\left(\frac{1}{2}\nabla^{2}\lambda\right)=\nu-\mu  \,\ \,\ \text{in}\,\ \mathbb{R}^{d},\label{1.8}
\end{equation}
which means that $\lambda \in \mathcal{M}(\mathbb{R}^{d}, S^{+}_{d})$, $|\lambda|(\mathbb{R}^{d})<+\infty$ and
\begin{equation}
\int_{\mathbb{R}^{d}}\frac{1}{2}\nabla^{2}\varphi(x):d \lambda(x) = \int_{\mathbb{R}^{d}}\varphi(x) \diff\nu(x)- \int_{\mathbb{R}^{d}}\varphi(x)\diff \mu(x) \,\ \,\ \forall \varphi \in C^{2}_{b}(\mathbb{R}^{d}). \label{1.9}
\end{equation}

We denote by $C_{b}((0,1)\times \mathbb{R}^{d}, S_{d})$ the space of all bounded continuous functions on $(0,1)\times \mathbb{R}^{d}$ with values into $S_{d}$. 
\begin{prop}\label{prop: 1.7} For each $\mu, \, \nu \in \mathcal{P}_{2}(\mathbb{R}^{d})$ the following holds
\begin{equation}
F(\mu,\nu)=\inf\left\{\int_{\mathbb{R}^{d}}f\left(\frac{\diff \lambda}{\diff |\lambda|}\right)\diff |\lambda| \st \tr\left(\frac{1}{2}\nabla^{2}\lambda\right)=\nu-\mu \right\}. \label{1.10}
\end{equation}
\end{prop}
\begin{rem} \label{rem: 1.8} Since $f$ is positively 1-homogeneous, for each positive finite measure $m$ on $\mathbb{R}^{d}$ such that $|\lambda| \ll m$, it holds
\[
\int_{\mathbb{R}^{d}}f\left(\frac{\diff \lambda}{\diff m}\right) \diff m = \int_{\mathbb{R}^{d}}f\left(\frac{\diff \lambda}{\diff |\lambda|}\right)\diff |\lambda|.
\]
\end{rem}
\begin{proof}[Proof of Proposition \ref{prop: 1.7}.] Let $\lambda \in \mathcal{M}(\mathbb{R}^{d}, S^{+}_{d})$ be a solution to (\ref{1.8}). For each $t\in [0,1]$, define
\begin{equation}
\varrho_{t}=(1-t)\mu+t\nu \in \mathcal{P}_{2}(\mathbb{R}^{d}). \label{1.11}
\end{equation}
Fix an arbitrary $\varphi \in C^{1,2}_{b}((0,1)\times \mathbb{R}^{d})$ with (closed) support in $(0,1) \times \mathbb{R}^{d}$. Then, using \eqref{1.11}, Fubini's theorem and integrating by parts, we have
\begin{equation}
\int^{1}_{0}\int_{\mathbb{R}^{d}}\partial_{t}\varphi(t,x) \diff \varrho_{t}(x) \diff t=\int_{\mathbb{R}^{d}}\int^{1}_{0}\varphi(t,x)\diff t \diff \mu(x)- \int_{\mathbb{R}^{d}}\int^{1}_{0}\varphi(t,x)\diff t\diff \nu(x). \label{1.12}
\end{equation}
Since $\tr(\frac{1}{2}\nabla^{2} \lambda)=\nu-\mu$ and $\int^{1}_{0}\varphi(t,\cdot)\diff t \in C^{2}_{b}(\mathbb{R}^{d})$, using (\ref{1.9}) and Fubini's theorem, we deduce that 
\begin{equation}
\int^{1}_{0}\int_{\mathbb{R}^{d}}\frac{1}{2}\nabla^{2}_{x}\varphi(t,x):\diff \lambda(x)\diff t= \int_{\mathbb{R}^{d}}\int^{1}_{0}\varphi(t,x)\diff t\diff \nu( x)-\int_{\mathbb{R}^{d}}\int^{1}_{0}\varphi(t,x)\diff t \diff \mu(x). \label{1.13}
\end{equation}
Combining (\ref{1.12}) and (\ref{1.13}),  we get
\[
\int^{1}_{0}\int_{\mathbb{R}^{d}}\partial_{t}\varphi(t,x)\diff \varrho_{t}(x)\diff t= -\int^{1}_{0}\int_{\mathbb{R}^{d}}\frac{1}{2}\nabla^{2}_{x}\varphi(t,x):\diff \lambda(x)\diff t,
\]
and hence the pair $(\varrho_{t}\otimes \mathcal{L}^{1}\mres (0,1),\lambda \otimes \mathcal{L}^{1} \mres (0,1))$ is a solution to (\ref{1.1}). Thus,
\begin{equation}
F(\mu,\nu)\leq \inf\left\{\int_{\mathbb{R}^{d}}f\left(\frac{\diff \lambda}{\diff |\lambda|}\right) \diff |\lambda| \st \tr\left(\frac{1}{2}\nabla^{2}\lambda\right)=\nu-\mu \right\}. \label{1.14}
\end{equation}

Now assume that $F(\mu,\nu)<+\infty$ and let $(\varrho_{t}\otimes \mathcal{L}^{1}\mres (0,1),\lambda_{t}\otimes \mathcal{L}^{1}\mres (0,1))$ be an admissible pair for (\ref{1.7}). Using (\ref{1.4}) with $t_{1}=0$ and $t_{2}=1$, for each $\varphi \in C^{2}_{b}(\mathbb{R}^{d})$, we obtain
\begin{equation} \label{form}
\int^{1}_{0}\int_{\mathbb{R}^{d}}\frac{1}{2}\nabla^{2}\varphi(x):\diff \lambda_{t}(x)\diff t =\int_{\mathbb{R}^{d}}\varphi(x)\diff \nu(x) - \int_{\mathbb{R}^{d}}\varphi(x)\diff \mu(x).
\end{equation}
Define $$\lambda=\pi^{x}_{\#}(\lambda_{t} \otimes \mathcal{L}^{1}\mres (0,1)) \in \mathcal{M}(\mathbb{R}^{d}, S^{+}_{d}),$$ where $\pi^{x}(t,x)=x$ for each $(t,x)\in [0,1]\times \mathbb{R}^{d}$. Then $|\lambda|(\mathbb{R}^{d})<+\infty$ and for each $\xi\in C_{b}(\mathbb{R}^{d}, S_{d})$,
\begin{equation}\label{eq 1.15}
\int_{\mathbb{R}^{d}}\xi(x):\diff \lambda(x)=\int^{1}_{0}\int_{\mathbb{R}^{d}}\xi(\pi^{x}(t,x)):\diff \lambda_{t}(x)\diff t.
\end{equation}
Gathering together \eqref{form} and \eqref{eq 1.15}, we get
\begin{equation*}
\int_{\mathbb{R}^{d}}\frac{1}{2}\nabla^{2}\varphi(x):\diff \lambda(x)=\int_{\mathbb{R}^{d}}\varphi(x)\diff \nu(x) - \int_{\mathbb{R}^{d}}\varphi(x)\diff \mu(x) \,\ \,\ \forall \varphi \in C^{2}_{b}(\mathbb{R}^{d}).
\end{equation*}
Thus, $\lambda$ is a solution to \eqref{1.8}. Let $\Psi$ be the convex function of Lemma~\ref{lemma conjugate prop}. According to Lemma~\ref{lemma conjugate prop}, \begin{align*}
\int_{\mathbb{R}^{d}}&f\left(\frac{\diff \lambda}{\diff| \lambda|}\right) \diff |\lambda|= \sup\left\{\int_{\mathbb{R}^{d}}\xi(x):\diff \lambda(x) \st \xi \in C_{b}(\mathbb{R}^{d}, S_{d}),\; \xi(\mathbb{R}^{d}) \subset \dom(f^{*})  \right\}\\&
=\sup\left\{\int^{1}_{0}\int_{\mathbb{R}^{d}}\xi(\pi^{x}(t,x)):\diff \lambda_{t}(x)\diff t \st \xi \in C_{b}(\mathbb{R}^{d}, S_{d}), \; \xi(\mathbb{R}^{d}) \subset \dom(f^{*}) \right\}\\
&\leq\sup\left\{\int^{1}_{0}\int_{\mathbb{R}^{d}} \xi(t,x):\diff \lambda_{t}(x)\diff t \st \xi \in C_{b}((0,1)\times \mathbb{R}^{d}, S_{d}), \; \xi((0,1)\times \mathbb{R}^{d})\subset \dom(f^{*})\right\}\\
&\leq\int^{1}_{0}\int_{\mathbb{R}^{d}}f\left(\frac{\diff \lambda_{t}}{\diff |\lambda_{t}|}\right) \diff |\lambda_{t}|\diff t.
\end{align*}
Using this together with the fact that $\lambda$ is a solution to \eqref{1.8} and taking into account (\ref{1.14}), we deduce \eqref{1.10}, which  completes our proof of Proposition~\ref{prop: 1.7}.
\end{proof}
As a byproduct of the proof of Proposition~\ref{prop: 1.7}, we obtain the next
\begin{cor} \label{cor: 1.9} If the infimum in \eqref{1.7} is achieved on $(\varrho_{t} \otimes \mathcal{L}^{1} \mres (0,1), \lambda_{t} \otimes \mathcal{L}^{1} \mres (0,1))$, then the infimum in (\ref{1.10}) is achieved on $\lambda=\pi^{x}_{\#}(\lambda_{t}\otimes \mathcal{L}^{1}\mres (0,1))$. Conversely, if the infimum in (\ref{1.10}) is achieved on $\lambda$, then the infimum in (\ref{1.7}) is achieved on $(\varrho_{t} \otimes \mathcal{L}^{1} \mres (0,1), \lambda_{t} \otimes \mathcal{L}^{1} \mres (0,1))$, where $\varrho_{t}=(1-t)\mu+t\nu$ and $\lambda_{t}=\lambda$.
\end{cor}
\begin{rem}\label{rem about abscont}
In general, one cannot find a martingale $(X_{t})_{t \in [0,1]}$ (for the definition, see, for instance, Section~2 in \cite{Huesmann-Trevisan-2019}) with continuous paths whose marginals are $\varrho_{t}=(1-t)\mu+t \nu$ for $t \in [0,1]$. Indeed, if $\mu=\delta_{\frac{x+y}{2}}$ and $\nu=\frac{1}{2}(\delta_{x}+\delta_{y})$, the optimal curve $\varrho_{t}=(1-t)\mu+t \nu$ is absolutely continuous in the $1$-Wasserstein distance, but it is not absolutely continuous in the $q$-Wasserstein distance with $q>1$, which from the particle point of view means that the particles must jump from $\frac{x+y}{2}$ to $x$ or $y$ at a certain rate. This phenomenon occurs because, in contrast to \cite{Huesmann-Trevisan-2019}, in our case the cost $f$ is not $q$-admissible, which correlates with the fact that for a solution $(\varrho, \lambda)$ of the problem \eqref{1.7}, $\lambda$ does not have to be absolutely continuous with respect to $\varrho$ (the reader may also consult \cite[Remark~3.3]{MR2870228}). 
\end{rem}

\subsection{Existence and duality}
We introduce a dual formulation to \eqref{1.7} and prove the existence of a minimizer  when $F(\mu, \nu)<+\infty$. 
\begin{theorem} \label{prop: 1.13} For each $\mu,\, \nu \in \mathcal{P}_{2}(\mathbb{R}^{d})$ the following equality holds
\begin{equation}\label{Dual}
F(\mu,\nu)=\sup\left\{\langle\nu-\mu, \psi\rangle \st \psi \in C^{2}_{b}(\mathbb{R}^{d}), \; -f^{*}\left(\frac{1}{2}\nabla^{2}\psi \right)=0\,\ \text{in}\,\ \mathbb{R}^{d}\right\}.
\end{equation}
Moreover, if $F(\mu,\nu)<+\infty$, then the infimum in \eqref{1.7} (and in \eqref{1.10}) is actually a minimum. 
\end{theorem}
\begin{rem}\label{rem 1.12}
	The function $f$ is not strictly convex and, generally speaking, is not coercive (see, for instance, the example~\ref{notcoerc}). Notice that in \cite[Theorem~4.3]{Huesmann-Trevisan-2019} the strict convexity and $q$-coercivity (for some $q>1$; see Section~2 in \cite{Huesmann-Trevisan-2019} for the definition) of the cost function are used to prove the existence of a minimizer of the primal problem. In particular, the fact that the Fenchel conjugate of the cost function is finite on $\{t I_{d} \st t \geq 0\}$ (since the cost function in \cite{Huesmann-Trevisan-2019} is $q$-coercive) is used to prove that a minimizer in the Fenchel-Rockafellar duality theorem (see (4.5) in \cite[Theorem~4.3]{Huesmann-Trevisan-2019}) solves the Fokker-Planck equation, where the diffusion term is weighted accordingly with the mass.  In our setting, if $\dom(f^{*})$ contained $\{t I_{d} \st t \geq 0\}$, we would have  $f= +\infty$ identically on $S^{++}_{d}$. It is also worth noting that our dual formulation \eqref{Dual} cannot be derived from the results of \cite{Tan-Touzi-2013}, established using stochastic control theory. Indeed, in view of Remark~\ref{rem about abscont}, the problem \eqref{1.7} and the semimartingale transportation problems studied in \cite{Tan-Touzi-2013} are different. Furthermore, the existence and duality results in \cite{Tan-Touzi-2013} are established under the crucial coercivity assumption (see Assumption~3.3 in \cite{Tan-Touzi-2013}), which guarantees tightness of any minimizing sequences of the (primal) problems in \cite{Tan-Touzi-2013} and does not hold for the cost functions behaving like $c(a) \sim |a|$ for $a \in S^{+}_{d}$, in particular, for our cost $f$ (notice also that the Assumption~8.1~(3) in \cite{10.3150/21-BEJ1346} does not hold for $f$). Thirdly, the coercivity of $f$ is not required in Theorem~\ref{prop: 1.13}.
	\end{rem}
\begin{proof}[Proof of Theorem~\ref{prop: 1.13}] Let $\Psi$ be the convex function of Lemma~\ref{lemma conjugate prop}. We write the dual pairing as $\sigma(\xi)$ for $\sigma \in (C_{b}(\mathbb{R}^{d}, S_{d}))^{\prime}$ and $\xi \in C_{b}(\mathbb{R}^{d}, S_{d})$. Then the following implication holds: if $\Psi^{**}(\sigma)<+\infty$, then $\sigma(\xi)\geq 0$ for each $\xi \in C_{b}(\mathbb{R}^{d}, S^{+}_{d})$. Assume by contradiction that $\Psi^{**}(\sigma)<+\infty$ and $\sigma(\xi)>0$ for some $\xi \in C_{b}(\mathbb{R}^{d}, S_{d})$ such that $-\xi \in C_{b}(\mathbb{R}^{d}, S^{+}_{d})$. Since $\Psi^{*}(t\xi)=0$ for each $t\geq 0$, $\Psi^{**}(\sigma)\geq t \sigma(\xi)>0$. Letting $t$ tend to $+\infty$, we obtain $\Psi^{**}(\sigma)=+\infty$, which leads to a contradiction and proves the above implication. 

Next, following \cite{Huesmann-Trevisan-2019}, we say that $\xi \in C_{b}(\mathbb{R}^{d}, S_{d})$ is represented by $\varphi \in C^{2}_{b}(\mathbb{R}^{d})$ if $\xi = -\frac{1}{2} \nabla^{2}\varphi$. We define $\Theta:C_{b}(\mathbb{R}^{d}, S_{d}) \to \mathbb{R} \cup \{+\infty\}$ by
\[
\Theta(\xi) = \begin{cases}
 \langle \mu -  \nu, \varphi \rangle \,\ & \text{if}\,\ \xi \,\ \text{is represented by}\,\ \varphi,\\
+\infty \,\ & \text{otherwise}.
\end{cases}
\]
If $\varphi_{1}$ and $\varphi_{2}$ represent $\xi$, then $\frac{1}{2}\nabla^{2}(\varphi_{1}-\varphi_{2})=0$, and hence $\varphi_{1}(x)=a+y \cdot x+ \varphi_{2}(x)$ for some fixed $a \in \mathbb{R}$ and $y \in \mathbb{R}^{d}$. Without loss of generality, we can assume that $\int_{\mathbb{R}^{d}}(a+y\cdot x) \diff\mu(x)=\int_{\mathbb{R}^{d}}(a+y \cdot x) \diff \nu(x)$, because otherwise $F(\mu,\nu)$ and the supremum in \eqref{Dual} are equal to $+\infty$: $F(\mu,\nu)=+\infty$ because $\mu$ and $\nu$ would not be in convex order; the supremum in \eqref{Dual} is equal to $+\infty$ by letting $\psi(x)=t y\cdot x$, which satisfies $-f^{*}(\frac{1}{2}\nabla^{2}\psi)=0$ in $\mathbb{R}^{d}$, and letting $t$ tend to $\pm \infty$ depending on the sign of the difference. This implies that $\langle \mu-\nu, \varphi_{1}-\varphi_{2}\rangle=0$. Thus, $\Theta$ does not depend on the choice of $\varphi$. It is also worth noting that the set of represented mappings $\xi \in C_{b}(\mathbb{R}^{d}, S_{d})$ is a linear subspace on which $\Theta$ is linear. In particular, $\Theta$ is positively $1$-homogeneous with the Fenchel conjugate
\[
\Theta^{*}(\sigma)=\sup\left\{\sigma(\xi) + \langle \nu -\mu, \varphi \rangle \st \xi \,\ \text{is represented}\right\}
\]
taking values in $\{0,+\infty\}$ and $\Theta^{*}(\sigma)=0$ if and only if 
\begin{equation}\label{dualpairingfpeextended}
\int_{\mathbb{R}^{d}} \frac{1}{2}\nabla^{2} \varphi : \diff \sigma = \int_{\mathbb{R}^{d}} \varphi \diff \nu - \int_{\mathbb{R}^{d}} \varphi \diff \mu \,\  \,\ \forall \varphi \in C^{2}_{b}(\mathbb{R}^{d}),
\end{equation}
where we interpret the integral on the left-hand side of the above equality as a duality pairing. 

We have proved that if $\Psi^{**}(\sigma)<+\infty$, then $\sigma$ is a nonnegative bounded linear functional. We claim that $\sigma$ is tight and hence induced by a measure. Let $g: \mathbb{R} \to [0,1]$ be a smooth nondecreasing function such that $g(t)=0$ for  $t \in (-\infty,1/2]$, $g(t)=1$ for  $t \in [1,+\infty)$ and $|g^{\prime}(t)|, |g^{\prime\prime }(t)| \leq 4$ for  $t \in \mathbb{R}$. Define the function $h:\mathbb{R} \to [0,+\infty)$ by $h(t)=\int_{-\infty}^{t} g(s)\diff s$. Then $h$ is a smooth function, $h^{\prime}(t)=g(t)$ and $h^{\prime\prime}(t)=g^{\prime}(t)\geq 0$ for $t \in \mathbb{R}$. For $L>0$ and $x \in \mathbb{R}^{d}$, we set $\varphi^{L}(x)=h(|x|^{2}-L^{2})$. Notice that $\xi^{L}(x)=-(g(|x|^{2}-L^{2})I_{d} + 2g^{\prime}(|x|^{2}-L^{2}) x\otimes x)$ is represented by $\varphi^{L}$ and $-\xi^{L} \geq I_{d}$ on $\mathbb{R}^{d} \setminus B_{\sqrt{L^{2}+1}}(0)$. Since $\nu \in \mathcal{P}_{2}(\mathbb{R}^{d})$ and $h(|x|^{2}-L^{2})\leq |x|^{2}-L^{2}\leq |x|^{2}$ on $\mathbb{R}^{d} \setminus B_{L}(0)$, for each fairly small $\varepsilon>0$, there exists $L>0$ large enough such that 
\begin{equation}\label{tightcondder}
\int_{\mathbb{R}^{d}}h(|x|^{2}-L^{2}) \diff \nu(x) < \varepsilon.
\end{equation}
For each $\xi \in C_{b}(\mathbb{R}^{d}, S_{d})$ such that $|\xi|\leq 1$ and $\supp(\xi) \subset \mathbb{R}^{d} \setminus B_{\sqrt{L^{2}+1}}(0)$, it holds $\xi \leq -\xi^{L}$ and $-\xi \leq -\xi^{L}$. Using this, the facts that $\sigma$ is a nonnegative linear functional and $\xi^{L}=-\frac{1}{2}\nabla^{2}\varphi^{L}$, \eqref{dualpairingfpeextended},  \eqref{tightcondder} and the fact that $\mu$ is a nonnegative measure, we deduce the following
\[
|\sigma(\xi)| \leq \sigma(-\xi^{L})=\int_{\mathbb{R}^{d}}h(|x|^{2}-L^{2})\diff (\nu(x)-\mu(x)) < \varepsilon,
\]
which proves that $\sigma$ is tight and hence induced by a measure. 

The mapping $I_{d}$ is represented by $\varphi(x)=-|x|^{2}$, $\Psi^{*}$ is continuous at $-I_{d}$ and $\Theta(I_{d})<+\infty$. After all, applying the formula for the conjugate of the sum $\Psi^{*}(-\cdot)+\Theta(\cdot)$ at $0 \in (C_{b}(\mathbb{R}^{d}, S_{d}))^{\prime}$ (see, for instance, \cite[Proposition~2.3~$(i)$]{BOUCHITTE2006642} or \cite[Theorem~1.12]{MR2759829}), we obtain 
\begin{equation}\label{dualitytheoremhb}
\inf\{\Psi^{**}(\sigma) + \Theta^{*}(\sigma) \st \sigma \in (C_{b}(\mathbb{R}^{d}, S_{d}))^{\prime}\}=\sup\{-\Psi^{*}(-\xi)-\Theta(\xi) \st \xi \in C_{b}(\mathbb{R}^{d}, S_{d})\},
\end{equation}
where the infimum is actually a minimum if the supremum, coinciding with the supremum in \eqref{Dual}, is finite. The latter holds if and only if $F(\mu,\nu)<+\infty$. Indeed, if the supremum in \eqref{dualitytheoremhb} is finite, then according to \cite[Proposition~2.3~$(ii)$]{BOUCHITTE2006642} (or \cite[Theorem~1.12]{MR2759829}), the infimum in \eqref{dualitytheoremhb} is actually a minimum and if $\sigma$ is a minimizer, then we have proved that $\sigma \in\mathcal{M}(\mathbb{R}^{d}, S^{+}_{d})$, $|\sigma|(\mathbb{R}^{d})<+\infty$, $\sigma$ solves \eqref{1.8} for $\mu$ and $\nu$ (see \eqref{dualpairingfpeextended})  and $\Psi(\sigma)=\Psi^{**}(\sigma)$. This, together with  Proposition~\ref{prop: 1.7}, implies that the left-hand side in \eqref{dualitytheoremhb} coincides with $F(\mu,\nu)$, and the infimum in \eqref{1.7} (and in \eqref{1.10}) is actually a minimum. On the other hand, if $F(\mu,\nu)<+\infty$, then there exists $\sigma \in \mathcal{M}(\mathbb{R}^{d}, S^{+}_{d})$ solving \eqref{1.8} for $\mu$ and $\nu$.  For each $\psi\in C^{2}_{b}(\mathbb{R}^{d})$ such that $-f^{*}(\frac{1}{2}\nabla^{2} \psi)=0$ in $\mathbb{R}^{d}$, we have $\frac{1}{2} \nabla^{2}\psi :\frac{\diff \sigma}{\diff |\sigma|} \le f(\frac{\diff \sigma}{\diff |\sigma|})$ $|\sigma|$-a.e. on $\mathbb{R}^{d}$. Hence  
\[
\langle\nu-\mu,  \psi\rangle=\int_{\mathbb{R}^{d}}\frac{1}{2}\nabla^{2}\psi :\diff \sigma  \le \int_{\mathbb{R}^{d}} f\left(\frac{\diff \sigma}{\diff |\sigma|}\right) \diff |\sigma|,
\]
which implies that the supremum in \eqref{dualitytheoremhb} is less than or equal to $\int_{\mathbb{R}^{d}} f(\frac{\diff \sigma}{\diff |\sigma|}) \diff |\sigma|<+\infty$. 
This completes our proof of Theorem~\ref{prop: 1.13}.
\end{proof}

\begin{cor} \label{cor: 1.15} For each $\mu, \nu \in \mathcal{P}_{2}(\mathbb{R}^{d})$ the following estimate holds
	\begin{equation}\label{eq. 1.21}
	F(\mu,\nu) \geq f\left(\int_{\mathbb{R}^{d}} x \otimes x (\diff \nu(x) - \diff \mu(x)) \right).
	\end{equation}
\end{cor}
\begin{proof}[Proof of Corollary \ref{cor: 1.15}] Let $A \in \dom(f^{*})$ be arbitrary and $\psi (x)=Ax\cdot x$ for each $x\in \mathbb{R}^{d}$. Then we have $-f^{*}(\frac{1}{2}\nabla^{2}\psi)=-f^{*}(A)  =0$ in $\mathbb{R}^{d}$, which, in view of Theorem~\ref{prop: 1.13}, yields
\begin{equation}\label{eq. 1.22}
F(\mu,\nu) \geq \int_{\mathbb{R}^{d}} Ax\cdot x \diff \nu(x) - \int_{\mathbb{R}^{d}} Ax\cdot x \diff \mu(x) =A:\left(\int_{\mathbb{R}^{d}} x \otimes x \diff \nu(x) - \int_{\mathbb{R}^{d}} x \otimes x \diff \mu(x)\right).   
\end{equation}
Since $A\in \dom(f^{*})$ was arbitrarily chosen, $f^{*}=0$ on $\dom(f^{*})$ and $f=f^{**}$ (this comes from \eqref{assumpsublinearf}, \eqref{assumplscf} and \cite[Theorem~2.1~$(i)$]{BOUCHITTE2006642}), (\ref{eq. 1.22}) implies (\ref{eq. 1.21}), which completes our proof of Corollary~\ref{cor: 1.15}.

\end{proof}

\subsection{Lower semicontinuity, convexity and subadditivity}
Hereinafter, we use the notation $\Phi_{ 2}(\mathbb{R}^{d})$ for the space of all functions $\varphi \in C(\mathbb{R}^{d})$ such that there exists a constant $C>0$ such that $|\varphi(x)|\leq C(1+|x|^{2})$ for each $x \in \mathbb{R}^{d}$. 
\begin{prop} \label{prop: 1.16} The following assertions hold.
\begin{enumerate}[label=(\roman*)]
\item \label{item 1 lcs} $F$ is convex on $\mathcal{P}_{2}(\mathbb{R}^{d})\times \mathcal{P}_{2}(\mathbb{R}^{d})$ and lower semicontinuous with respect to the weak topology on $\mathcal{P}_{2}(\mathbb{R}^{d})\times \mathcal{P}_{2}(\mathbb{R}^{d})$ in duality with $\Phi_{2}(\mathbb{R}^{d})\times \Phi_{2}(\mathbb{R}^{d})$.
\item \label{item 2 lcs} For each choice of $\mu_{1},\, \mu_{2},\, \mu_{3} \in \mathcal{P}_{2}(\mathbb{R}^{d})$,
\[
F(\mu_{1}, \mu_{3}) \leq F(\mu_{1}, \mu_{2}) + F(\mu_{2}, \mu_{3}).
\]
\end{enumerate}
\end{prop}
\begin{proof} It is a direct consequence of Theorem~\ref{prop: 1.13} that $F$ is convex and lower semicontinuous with respect to the specified product topology, since it is represented as the supremum of the family consisting of linear functionals that are continuous with respect to this topology. This proves $\ref{item 1 lcs}$.

Let us prove $\ref{item 2 lcs}$. For each $\psi \in C^{2}_{b}(\mathbb{R}^{d})$ such that $-f^{*}(\frac{1}{2} \nabla^{2}\psi)=0$ in $\mathbb{R}^{d}$, using Theorem~\ref{prop: 1.13}, we have
\[
\langle \mu_{3}-\mu_{1}, \psi\rangle = \langle \mu_{2}-\mu_{1}, \psi \rangle + \langle \mu_{3}-\mu_{2}, \psi \rangle \le F(\mu_{1}, \mu_{2})+F(\mu_{2}, \mu_{3}),
\]
which implies $\ref{item 2 lcs}$ and completes our proof of Proposition~\ref{prop: 1.16}.
\end{proof}

\subsection{Associated weak transport problem}
We define the cost function $G:\mathbb{R}^{d} \times \mathcal{P}_{2}(\mathbb{R}^{d}) \to [0,+\infty]$ by
\begin{equation}
G(x,p)=F(\delta_{x},p). \label{1.16}
\end{equation}
By Proposition~\ref{prop: 1.16}, $G$ is lower semicontinuous in $(x,p)$ and convex in $p$. For each $\mu, \nu \in \mathcal{P}_{2}(\mathbb{R}^{d})$, we consider the following weak transport problem
\begin{equation*}
 \inf \left\{\int_{\mathbb{R}^{d}}G(x,\gamma^{x}) \diff \mu(x) \st \gamma \in \Pi(\mu,\nu) \right\}
\end{equation*}
and define the functional $H:\mathcal{P}_{2}(\mathbb{R}^{d})\times \mathcal{P}_{2}(\mathbb{R}^{d})\to [0,+\infty]$ by
\begin{equation}\label{thedefofH}
H(\mu,\nu)=\inf \left\{\int_{\mathbb{R}^{d}}G(x,\gamma^{x}) \diff \mu(x) \st \gamma \in \Pi(\mu,\nu) \right\}. 
\end{equation}
Since $G(x,\delta_{x})=0$ for each $x \in \mathbb{R}^{d}$ and $\gamma^{x}\otimes \mu\in \Pi(\mu,\mu)$ when $\gamma^{x}=\delta_{x}$ for $\mu$-a.e. $x \in \mathbb{R}^{d}$,
\begin{equation*}
H(\mu,\mu)=F(\mu,\mu)=0 \,\  \,\ \forall \mu \in \mathcal{P}_{2}(\mathbb{R}^{d}),
\end{equation*}
which implies that  the functional $H$ is proper (i.e., $\dom(H)\not = \emptyset$). 

We shall prove the equality $F=H$. First, we show that $H\geq F$, which is a consequence of Theorem~\ref{prop: 1.13}. To establish the converse inequality, which is a delicate matter, we develop the dual result of \cite{MR4029731} and the theory for subadditive costs that appeared in \cite[Section~6]{Alibert-Bouchitte-Champion-2019}, where the role of $\mathbb{R}^{d}$ is replaced by the closure of a bounded open convex subset of $\mathbb{R}^{d}$. The main difficulty is that, unlike \cite{Alibert-Bouchitte-Champion-2019}, we work with the set of probability measures on $\mathbb{R}^{d}$, which is not compact with respect to the weak topology. 

\begin{prop}\label{prop lower bound} For each $\mu, \nu \in \mathcal{P}_{2}(\mathbb{R}^{d})$, $H(\mu,\nu)\geq F(\mu,\nu)$.
\end{prop}
\begin{proof} Fix an arbitrary $\gamma=\gamma^{x}\otimes \mu \in \Pi(\mu,\nu)$. According to the definition of $G$ and Theorem~\ref{prop: 1.13}, for $\mu$-a.e. $x\in \mathbb{R}^{d}$  and for each $\psi \in C^{2}_{b}(\mathbb{R}^{d})$ such that $-f^{*}(\frac{1}{2}\nabla^{2}\psi)=0$ in $\mathbb{R}^{d}$, we have 
\begin{equation*}
G(x, \gamma^{x}) \geq \int_{\mathbb{R}^{d}}\psi(y) \diff \gamma^{x}(y) - \psi(x).
\end{equation*}
Integrating both sides of the above inequality over $\mathbb{R}^{d}$ with respect to $\mu$, we obtain
\begin{align*}
\int_{\mathbb{R}^{d}}G(x, \gamma^{x}) \diff \mu(x) &\geq \int_{\mathbb{R}^{d}} \int_{\mathbb{R}^{d}}\psi(y)\diff \gamma^{x}(y) \diff \mu (x) - \int_{\mathbb{R}^{d}}\psi(x) \diff \mu(x) 
= \langle \nu- \mu, \psi \rangle.
\end{align*}
This,  since $\gamma \in \Pi(\mu, \nu)$ and $\psi \in C^{2}_{b}(\mathbb{R}^{d})$ satisfying $-f^{*}(\frac{1}{2}\nabla^{2}\psi)=0$ in $\mathbb{R}^{d}$, according to Theorem~\ref{prop: 1.13}, completes our proof of Proposition~\ref{prop lower bound}.
\end{proof}

For each $x \in \mathbb{R}^{d}$ and for each universally measurable function $\varphi: \mathbb{R}^{d} \to \mathbb{R}$ satisfying the estimate $|\varphi(\cdot)|\leq C(1+|\cdot|^{2})$ for some constant $C>0$, we define
\begin{equation}\label{G-transform ABC}
\varphi^{G}(x):=\inf\left\{\int_{\mathbb{R}^{d}}\varphi\diff p + G(x,p) \st p \in \mathcal{P}_{2}(\mathbb{R}^{d})\right\}.
\end{equation}
Inasmuch as $G(x,\delta_{x})=0$,
\begin{equation}\label{eq ineq49050t5ti0i45t}
\varphi^{G}\leq \varphi.
\end{equation}
We denote by $\Phi_{bb,2}(\mathbb{R}^{d})$ the subset of functions in $\Phi_{2}(\mathbb{R}^{d})$ which are bounded from below. 
\begin{rem}\label{remark G->phi is welld}
If $\varphi \in \Phi_{bb, 2}(\mathbb{R}^{d})$, then $\varphi^{G}$ is lower semianalytic (and hence universally measurable, see \cite[Proposition~7.47]{MR0511544}), bounded from below, $|\varphi^{G}(\cdot)|\leq C(1+|\cdot|^{2})$ for some constant $C>0$ and the integral $\int_{\mathbb{R}^{d}} \varphi^{G} \diff p$ is well defined for all $p \in \mathcal{P}_{2}(\mathbb{R}^{d})$. In particular, for each $x \in \mathbb{R}^{d}$, we can define $\varphi^{GG}(x)$. 
\end{rem}

\begin{prop}\label{proposition 1.17}  The following assertions hold.
\begin{enumerate}[label=(\roman*)]
\item \label{H 1-poshomog} $H$ is convex on $\mathcal{P}_{2}(\mathbb{R}^{d})\times \mathcal{P}_{2}(\mathbb{R}^{d})$ and lower semicontinuous with respect to the weak topology on $\mathcal{P}_{2}(\mathbb{R}^{d})\times \mathcal{P}_{2}(\mathbb{R}^{d})$ in duality with $\Phi_{2}(\mathbb{R}^{d})\times \Phi_{2}(\mathbb{R}^{d})$. If $H(\mu,\nu)<+\infty$, then the weak transport problem \eqref{thedefofH} admits a solution. 
\item \label{dualformh} For each $\mu, \nu \in \mathcal{P}_{2}(\mathbb{R}^{d})$, 
\begin{equation}
 H(\mu,\nu)=\sup\left\{\int_{\mathbb{R}^{d}}\varphi^{G}\diff \mu - \int_{\mathbb{R}^{d}}\varphi\diff \nu \st \varphi \in \Phi_{bb,2}(\mathbb{R}^{d})\right\}. \label{1.29}
\end{equation}
\end{enumerate}
\end{prop}
\begin{proof}
 According to Proposition~\ref{prop: 1.16}~$\ref{item 1 lcs}$, $(x,p)\mapsto G(x,p)$ is convex in $p$ and lower semicontinuous in $(x,p)$ with respect to the product topology on $\mathbb{R}^{d} \times \mathcal{P}_{2}(\mathbb{R}^{d})$, where the topology on $\mathbb{R}^{d}$ is generated by the Euclidean distance and the topology on $\mathcal{P}_{2}(\mathbb{R}^{d})$ is generated by the $2$-Wasserstein distance (see Definition~6.8 and Theorem~6.9 in \cite{Villani-2009}).  Then, using \cite[Theorem~2.9]{MR4029731}, we deduce that $H$ is lower semicontinuous with respect to the weak topology on $(\mathcal{P}_{2}(\mathbb{R}^{d}))^{2}$ in duality with $(\Phi_{2}(\mathbb{R}^{d}))^{2}$ and prove that  \eqref{thedefofH} admits a solution whenever $H(\mu, \nu)<+\infty$. Applying \cite[Theorem~3.1]{MR4029731}, we obtain the dual formulation \eqref{1.29}, which implies the convexity  of $H$. This completes our proof of Proposition~\ref{proposition 1.17}.
\end{proof}

To develop the theory of subadditive cost functionals, which appeared earlier in \cite[Section~6]{Alibert-Bouchitte-Champion-2019}, where the role of $\mathbb{R}^{d}$ is replaced by the closure of a bounded open convex subset of $\mathbb{R}^{d}$, we need to introduce some additional assumptions on $f$, namely either the coercivity (see the example~\ref{thirdexample}), or the growth assumption (see the example~\ref{notcoerc}), or both (see the examples~\ref{tracefirstexample},~\ref{levsecondexample}). In particular, we introduce the following assumptions.
\begin{align}
& \,\ \,\ \text{$f$ is coercive}. \tag{$A3$} \label{assumpcoercivef}\\
& \,\ \,\  \dom(f)=S^{+}_{d} \,\ \text{and there exists}\,\ \kappa_{1}>0 \,\ \text{such that}\,\  f(A)\leq \kappa_{1} \tr(A) \,\ \text{for all} \,\ A \in S^{+}_{d}.  \tag{$A4$} \label{assumpgrowthf}
\end{align}
\begin{rem}\label{equiv rem coer}
Since $f$ satisfies \eqref{assumpsublinearf} and \eqref{assumplscf}, the assumption \eqref{assumpcoercivef} holds if and only if there exists a constant $\kappa_{0}>0$ such that $f(A)\geq \kappa_{0} \tr(A)$ for each $A \in S^{+}_{d}$. Indeed, assume that \eqref{assumpcoercivef} holds. If $A\in S^{+}_{d}$ and $A \not =0 $, then, in view of \eqref{assumpsublinearf}, $f(A)=|A|f(\frac{A}{|A|})$. Using this, \eqref{assumpdom} and \eqref{assumplscf}, we deduce that there exists $\widetilde{E}\in S^{+}_{d}$ such that $|\widetilde{E}|=1$ and $f(\widetilde{E})=\min\{f(E) \st E\in S^{+}_{d},\; |E|=1 \}<+\infty$. Then, defining $\kappa_{0}=f(\widetilde{E})/\sqrt{d}$, we have $f(A) \geq f(\widetilde{E})|A| \geq  \kappa_{0} \tr(A)$ for each $A \in S^{+}_{d}$. Clearly,  the last inequality implies the coercivity of $f$.
\end{rem}

Next, we prove that $G$ is \textit{narrowly} lower semicontinuous if \eqref{assumpcoercivef} holds. Recall that $(p_{n})_{n\in \mathbb{N}}\subset \mathcal{P}(\mathbb{R}^{d})$ narrowly converges to $p \in \mathcal{P}(\mathbb{R}^{d})$ if $\int_{\mathbb{R}^{d}}\varphi \diff p_{n} \to \int_{\mathbb{R}^{d}} \varphi \diff p$ as $n\to +\infty$ for each $\varphi\in C_{b}(\mathbb{R}^{d})$. 
\begin{prop}\label{Fisnlsc}
Let \eqref{assumpcoercivef} hold,  $x_{n} \to  x \in \mathbb{R}^{d}$, $(p_{n})_{n\in \mathbb{N}} \subset \mathcal{P}_{2}(\mathbb{R}^{d})$ and $p_{n}$ narrowly converges to $p\in \mathcal{P}_{2}(\mathbb{R}^{d})$. Then 
\begin{equation*}\label{Gislscwrtnarrow}
G(x,p) \leq \liminf_{n\to +\infty} G(x_{n}, p_{n}).
\end{equation*}
\end{prop}
\begin{proof}
Without loss of generality, there exists a constant $C>0$ (independent of $n$) such that for each $n\in\mathbb{N}$, $G(x_{n}, p_{n}) \leq C$. Then, by Theorem~\ref{prop: 1.13}, the infimum in \eqref{1.10}, where $\mu=\delta_{x_{n}}$ and $\nu=p_{n}$, is actually a minimum. Thus, for each $n \in \mathbb{N}$, there exists $\lambda_{n} \in \mathcal{M}(\mathbb{R}^{d}, S^{+}_{d})$ such that $|\lambda_{n}|(\mathbb{R}^{d})<+\infty$, $\tr(\frac{1}{2} \nabla^{2} \lambda_{n})=p_{n}-\delta_{x_{n}}$ and $G(x_{n}, p_{n})=\int_{\mathbb{R}^{d}}f(\frac{\diff \lambda_{n}}{\diff| \lambda_{n}|}) \diff |\lambda_{n}|$. Using this and \eqref{assumpcoercivef}, we obtain
\begin{equation*}
|\lambda_{n}|(\mathbb{R}^{d})\leq \eta_{0}\int_{\mathbb{R}^{d}}f\left(\frac{\diff \lambda_{n}}{\diff |\lambda_{n}|}\right) \diff |\lambda_{n}| \leq \eta_{0}C
\end{equation*}
for some constant $\eta_{0}>0$ independent of $n$ (see Remark~\ref{equiv rem coer}). Then, according to the Banach-Alaoglu theorem, there exists $\lambda \in \mathcal{M}(\mathbb{R}^{d}, S^{+}_{d})$ such that $|\lambda|(\mathbb{R}^{d}) <+\infty$ and $\lambda_{n}$ converges weakly to $\lambda$. Since, for each $\varphi \in C^{2}_{c}(\mathbb{R}^{d})$,
\[
\int_{\mathbb{R}^{d}} \varphi \diff p_{n} -\varphi(x_{n})=\int_{\mathbb{R}^{d}} \frac{1}{2}\nabla^{2} \varphi :\diff \lambda_{n}, 
\]
letting $n$ tend to $+\infty$ and using the weak convergences, we deduce that
\begin{equation}\label{densityargument777}
\int_{\mathbb{R}^{d}} \varphi \diff p -\varphi(x)=\int_{\mathbb{R}^{d}} \frac{1}{2}\nabla^{2} \varphi :\diff \lambda.
\end{equation}
By direct adaptation of the density argument in \cite[Remark~2.3]{Trevisan-2016}, \eqref{densityargument777} implies that $\tr(\frac{1}{2}\nabla^{2}\lambda)=p-\delta_{x}$. Thus, $\lambda$ is a competitor for $G(x,p)=F(\delta_{x},p)$ (see \eqref{1.10}), which, in view of the lower semicontinuity of the function $\sigma \mapsto \int_{\mathbb{R}^{d}}f(\frac{\diff \sigma}{\diff |\sigma|})\diff |\sigma|$ on the subset of finite measures in $\mathcal{M}(\mathbb{R}^{d}, S^{+}_{d})$, yields the estimate
\[
G(x,p)\leq \int_{\mathbb{R}^{d}}f\left(\frac{\diff \lambda}{\diff |\lambda|}\right)\diff |\lambda| \leq \liminf_{n\to +\infty} \int_{\mathbb{R}^{d}}f\left(\frac{\diff \lambda_{n}}{\diff |\lambda_{n}|}\right)\diff |\lambda_{n}| = \liminf_{n\to +\infty} G(x_{n}, p_{n})
\]
and completes our proof of Proposition~\ref{Fisnlsc}.
\end{proof}
Let $\mathcal{S}_{bb,2}(\mathbb{R}^{d})$ be the set of all lower semicontinuous functions $\varphi:\mathbb{R}^{d}\to \mathbb{R}$ such that $\varphi$ is bounded from below and for some constant $C>0$, $|\varphi(x)|\leq C(1+|x|^{2})$ for each $x \in \mathbb{R}^{d}$.
\begin{prop}\label{prop lsc for G} Let \eqref{assumpcoercivef} hold and $\varphi\in \mathcal{S}_{bb,2}(\mathbb{R}^{d})$. Then the infimum for $\varphi$ in \eqref{G-transform ABC} is actually a minimum and $\varphi^{G} \in \mathcal{S}_{bb,2}(\mathbb{R}^{d})$. 
\end{prop}
\begin{proof}
Let $x \in \mathbb{R}^{d}$ and $(p_{n})_{n\in \mathbb{N}} \subset \mathcal{P}_{2}(\mathbb{R}^{d})$ be a minimizing sequence for $\varphi^{G}(x) \in \mathbb{R}$. Then there exists a constant $C>0$ (independent of $n$) such that for each $n\in \mathbb{N}$ large enough, $G(x,p_{n})\leq C$.  Using this, together with \eqref{assumpcoercivef} (see Remark~\ref{equiv rem coer}) and Corollary~\ref{cor: 1.15}, we deduce that 
\begin{equation*}
A:=\sup_{n\in \mathbb{N}} \int_{\mathbb{R}^{d}}|y|^{2}\diff p_{n}(y)<+\infty,
\end{equation*}
which, in view of \cite[Remark~5.1.5]{Ambrosio-Gigli-Savare-2008}, implies that $(p_{n})_{n\in \mathbb{N}}$ is tight. Then, by Prokhorov's theorem (see \cite[Theorem~5.1.3]{Ambrosio-Gigli-Savare-2008}), there exists a probability measure $p$ on $\mathbb{R}^{d}$ such that, up to a subsequence (not relabeled), $p_{n}$ converges narrowly to $p$. Thus, \[\int_{\mathbb{R}^{d}} |y|^{2} \diff p(y) \leq \liminf_{n\to +\infty} \int_{\mathbb{R}^{d}} |y|^{2}\diff p_{n}(y) \leq A\]and $p \in \mathcal{P}_{2}(\mathbb{R}^{d})$. By the narrow convergence (recall that $\varphi \in \mathcal{S}_{bb,2}(\mathbb{R}^{d})$) and Proposition~\ref{Fisnlsc}, 
\[
\int_{\mathbb{R}^{d}}\varphi \diff p + G(x,p) \leq \liminf_{n\to +\infty} \int_{\mathbb{R}^{d}} \varphi \diff p_{n} + G(x, p_{n}),
\]
which says that $p$ is a minimizer in the definition of $\varphi^{G}(x)$ (see \eqref{G-transform ABC}). 

Next, we prove the lower semicontinuity of $\varphi^{G}$. Let  $x_{n} \to x$,  $p_{n} \in \mathcal{P}_{2}(\mathbb{R}^{d})$ be a minimizer in the definition of $\varphi^{G}(x_{n})$ and $\liminf_{n\to+\infty} \varphi^{G}(x_{n})<+\infty$. Proceeding as before, we can assume that there exists $p \in \mathcal{P}_{2}(\mathbb{R}^{d})$ such that, up to a subsequence (not relabeled), $p_{n}$  converges narrowly to $p$. Using the narrow convergence, the fact that $\varphi \in \mathcal{S}_{bb,2}(\mathbb{R}^{d})$ and Proposition~\ref{Fisnlsc}, we obtain the following
\[
\varphi^{G}(x)\leq \int_{\mathbb{R}^{d}}\varphi \diff p + G(x,p) \leq \liminf_{n\to+\infty} \int_{\mathbb{R}^{d}}\varphi \diff p_{n} + G(x_{n}, p_{n})=\liminf_{n\to +\infty} \varphi^{G}(x_{n}),
\]
which proves the lower semicontinuity of $\varphi^{G}$ and completes our proof of Proposition~\ref{prop lsc for G}.
\end{proof}
Under the assumption \eqref{assumpcoercivef}, we can, on the one hand relax and, on the other hand, strengthen the dual constraint in \eqref{1.29} using bounded lower semicontinuous functions. 
\begin{prop}\label{prop relax H}
Let $\mu, \nu \in \mathcal{P}_{2}(\mathbb{R}^{d})$ and \eqref{assumpcoercivef} hold. Then
\begin{equation}\label{duality for H through Sb}
H(\mu, \nu)=\sup\left\{\int_{\mathbb{R}^{d}} \varphi^{G} \diff \mu -\int_{\mathbb{R}^{d}}\varphi \diff \nu \st \varphi \in \mathcal{S}_{b}(\mathbb{R}^{d}) \right\}.
\end{equation}
\end{prop}
\begin{proof}
Let $\varphi\in \Phi_{bb,2}(\mathbb{R}^{d})$ and  $\varphi_{n}=\min\{n, \varphi\}$ for each $n\in \mathbb{N}$. Then $\varphi_{n} \in C_{b}(\mathbb{R}^{d})$ and $\varphi_{n} \nearrow \varphi$ as $n \to +\infty$. According to Proposition~\ref{prop lsc for G}, for each $n \in \mathbb{N}$ and for each $x\in \mathbb{R}^{d}$, there exists $p_{n} \in \mathcal{P}_{2}(\mathbb{R}^{d})$ such that $\varphi^{G}_{n}(x)=\int_{\mathbb{R}^{d}} \varphi_{n} \diff p_{n} + G(x, p_{n})$. Since $\varphi^{G}_{n}(x)\leq \varphi(x)<+\infty$, arguing by the same way as in the proof of Proposition~\ref{prop lsc for G}, we deduce that there exists $p \in \mathcal{P}_{2}(\mathbb{R}^{d})$ such that, up to a subsequence (not relabeled), $p_{n}$ converges narrowly to $p$. For each $k \in \mathbb{N}$, using the weak convergence and Proposition~\ref{Fisnlsc}, we obtain 
\[
\int_{\mathbb{R}^{d}} \varphi_{k} \diff p +G(x,p)\leq \liminf_{n\to +\infty} \int_{\mathbb{R}^{d}}\varphi_{k} \diff p_{n} + G(x, p_{n}) \leq \liminf_{n\to +\infty} \int_{\mathbb{R}^{d}} \varphi_{n} \diff p_{n} + G(x, p_{n})=\liminf_{n \to +\infty} \varphi^{G}_{n}(x).
\]
Letting $k$ tend to $+\infty$, by the monotone convergence theorem, we have  
\[
\varphi^{G}(x) \leq \int_{\mathbb{R}^{d}} \varphi\diff p +G(x,p)=\lim_{k\to +\infty} \int_{\mathbb{R}^{d}} \varphi_{k} \diff p + G(x,p)\leq \liminf_{n\to +\infty} \varphi^{G}_{n}(x). 
\]
On the other hand, $\varphi_{n}^{G}(x)\leq \varphi^{G}(x)$ for each $n \in \mathbb{N}$ and hence $\varphi_{n}^{G}(x) \nearrow \varphi^{G}(x)$ as $ n\to +\infty$. Thus, by the monotone convergence theorem, $ \int_{\mathbb{R}^{d}} \varphi^{G}_{n} \diff \mu \to \int_{\mathbb{R}^{d}} \varphi^{G} \diff \mu$ and $ \int_{\mathbb{R}^{d}}\varphi_{n} \diff \nu \to \int_{\mathbb{R}^{d}} \varphi \diff \nu$ as $n\to +\infty$. This, together with \eqref{1.29}, implies that
\[
H(\mu, \nu)=\sup\left \{\int_{\mathbb{R}^{d}}\varphi^{G} \diff \mu - \int_{\mathbb{R}^{d}} \varphi\diff \nu \st \varphi \in C_{b}(\mathbb{R}^{d})\right\}.
\]
Since for each $\varphi \in \mathcal{S}_{b}(\mathbb{R}^{d})$, there exists a sequence $(\varphi_{n})_{n\in \mathbb{N}} \subset C_{b}(\mathbb{R}^{d})$ such that $\varphi_{n} \nearrow \varphi$ as $n\to +\infty$, repeating the above procedure, we complete our proof of Proposition~\ref{prop relax H}.
\end{proof}
If the $G$-transform $\varphi \mapsto \varphi^{G}$ is idempotent on $\mathcal{S}_{b}(\mathbb{R}^{d})$, the following dual formulation for $H(\mu, \nu)$ holds.
\begin{prop}\label{prop dual HI}
Let $\mu, \nu \in \mathcal{P}_{2}(\mathbb{R}^{d})$, \eqref{assumpcoercivef} hold and $\varphi^{GG}=\varphi^{G}$ for each $\varphi \in \mathcal{S}_{b}(\mathbb{R}^{d})$. Then
\begin{equation}\label{dualthroughG}
H(\mu, \nu)=\sup\{\langle \nu-\mu, \psi\rangle \st \psi \in \mathcal{U}_{b}(\mathbb{R}^{d}),\; -\psi = (-\psi)^{G}\}.
\end{equation}
\end{prop}
\begin{proof}
In view of Proposition~\ref{prop relax H}, $H(\mu, \nu)$ is greater than or equal to the supremum in \eqref{dualthroughG}. On the other hand, using the estimate $\varphi^{G} \leq \varphi$ in \eqref{duality for H through Sb}, we deduce the following
\begin{align*}
H(\mu, \nu) \leq \sup\{\langle \mu -\nu, \varphi^{G} \rangle \st \varphi \in \mathcal{S}_{b}(\mathbb{R}^{d})\}\leq \sup\{\langle \mu -\nu, \varphi \rangle \st \varphi \in \mathcal{S}_{b}(\mathbb{R}^{d}), \; \varphi=\varphi^{G}\},
\end{align*}
where the latter estimate comes from the assumption that $\varphi^{GG}=\varphi^{G}$ for each $\varphi \in\mathcal{S}_{b}(\mathbb{R}^{d})$, since in this case, taking into account Proposition~\ref{prop lsc for G}, we have $\{\varphi^{G} \st \varphi\in \mathcal{S}_{b}(\mathbb{R}^{d})\}\subset \{\varphi \st \varphi \in \mathcal{S}_{b}(\mathbb{R}^{d}),\; \varphi=\varphi^{G}\}$. Thus, the supremum in \eqref{dualthroughG} is greater than or equal to $H(\mu, \nu)$ and the dual formulation \eqref{dualthroughG} holds, which completes our proof of Proposition~\ref{prop dual HI}.
\end{proof}
The following proposition, which is a generalization of Proposition~6.4 in \cite{Alibert-Bouchitte-Champion-2019}, describes some situations in which the $G$-transform is idempotent on $\mathcal{S}_{b}(\mathbb{R}^{d})$, which in particular happens when $H$ is subadditive.
\begin{prop}\label{prop equiv IMS}
Let  \eqref{assumpcoercivef} hold. Then the following assertions are equivalent.
\begin{enumerate}[label=(\roman*)]
\item \label{Sb assertion I} For each choice of $\mu, \nu, p \in \mathcal{P}_{2}(\mathbb{R}^{d})$, $H(\mu, \nu)\leq H(\mu, p) + H(p, \nu)$.
\item \label{Sb assertion II} For each $\nu, p \in \mathcal{P}_{2}(\mathbb{R}^{d})$ and analytically measurable probability kernel $y \in \mathbb{R}^{d}\mapsto \gamma^{y}\in \mathcal{P}_{2}(\mathbb{R}^{d})$,
\begin{equation}
G(x,\nu) \leq G(x,p) + \int_{\mathbb{R}^{d}} G(y, \gamma^{y}) \diff p(y) \,\ \,\ \text{whenever}\,\ \nu=\int_{\mathbb{R}^{d}}\gamma^{y} \diff p(y).
\end{equation}
\item \label{Sb assertion III} For each $\varphi \in \mathcal{S}_{b}(\mathbb{R}^{d})$, $\varphi^{GG}=\varphi^{G}$.
\end{enumerate}
\end{prop}
\begin{proof}
The proof of the implication $\ref{Sb assertion I}\Rightarrow \ref{Sb assertion II}$ follows by choosing $\mu=\delta_{x}$, $\nu=\int_{\mathbb{R}^{d}}\gamma^{y}\diff p(y)$ and using the definition of $H(p,\nu)$ as the infimum (see \eqref{thedefofH}).  

Now we prove the implication $\ref{Sb assertion II}\Rightarrow \ref{Sb assertion III}$. For each $\varphi \in \mathcal{S}_{b}(\mathbb{R}^{d})$, in view of Proposition~\ref{prop lsc for G}, which applies in particular to every function lying in $\mathcal{S}_{b}(\mathbb{R}^{d}) \subset \mathcal{S}_{bb,2}(\mathbb{R}^{d})$, and inasmuch as $\varphi^{G}\leq \varphi$ (see \eqref{eq ineq49050t5ti0i45t}),  we have $\varphi^{G} \in \mathcal{S}_{b}(\mathbb{R}^{d})$. Repeating this observation for $ \varphi^{G} \in \mathcal{S}_{b}(\mathbb{R}^{d})$, one can see that $\varphi^{GG} \in \mathcal{S}_{b}(\mathbb{R}^{d})$  and $\varphi^{GG}\leq \varphi^{G}$. Thus, it is enough to prove the estimate
\begin{equation}\label{optimalitephiG}
\varphi^{G}(x)\leq \int_{\mathbb{R}^{d}}\varphi^{G}(y)\diff p(y) + G(x,p)
\end{equation}
for each $x\in \mathbb{R}^{d}$ and $p \in \mathcal{P}_{2}(\mathbb{R}^{d})$.
Using the lower semicontinuity of $G$ and the fact that $\varphi \in \mathcal{S}_{b}(\mathbb{R}^{d})$, we deduce that the function $(x,p)\mapsto \int_{\mathbb{R}^{d}}\varphi\diff p + G(x,p)$ is lower semicontinuous on $\mathbb{R}^{d}\times \mathcal{P}_{2}(\mathbb{R}^{d})$, where the topology on $\mathbb{R}^{d}$ is generated by the Euclidean distance and the topology on $\mathcal{P}_{2}(\mathbb{R}^{d})$ is generated by the 2-Wasserstein distance. Then, according to \cite[Proposition~7.50]{MR0511544}, for each $\varepsilon>0$ there exists an analytically measurable probability kernel $y \in \mathbb{R}^{d}\mapsto \gamma^{y} \in \mathcal{P}_{2}(\mathbb{R}^{d})$ such that
\begin{equation*}\label{mfor3fi034j0ifj0i}
\varphi^{G}(y)+\varepsilon\geq \int_{\mathbb{R}^{d}}\varphi(z) \diff \gamma^{y}(z) + G(y, \gamma^{y}).
\end{equation*}
Then defining $\nu(d z)= \int_{\mathbb{R}^{d}}\gamma^{y}(d z) \diff p(y)$, integrating both sides of the above inequality with respect to $p\in \mathcal{P}_{2}(\mathbb{R}^{d})$ and using $\ref{Sb assertion II}$, we obtain
\begin{align*}
G(x,p) + \int_{\mathbb{R}^{d}}\varphi^{G}(y)\diff p(y) +\varepsilon &\geq G(x,p)+\int_{\mathbb{R}^{d}} G(y, \gamma^{y}) \diff p(y) + \int_{\mathbb{R}^{d}}\int_{\mathbb{R}^{d}} \varphi(z)\diff \gamma^{y}(z)\diff p(y)\\
&\geq G(x,\nu) + \int_{\mathbb{R}^{d}}\varphi(z)\diff \nu(z)\\
&\geq \varphi^{G}(x),
\end{align*}
which yields \eqref{optimalitephiG} and completes our proof of the implication $\ref{Sb assertion II}\Rightarrow \ref{Sb assertion III}$.

The implication $\ref{Sb assertion III}\Rightarrow \ref{Sb assertion I}$ is a direct consequence of Proposition~\ref{prop dual HI}, since for each $\psi \in \mathcal{U}_{b}(\mathbb{R}^{d})$ such that $-\psi=(-\psi)^{G}$,
\[
\langle \nu-\mu, \psi\rangle\leq \langle p-\mu, \psi\rangle+ \langle\nu-p, \psi\rangle\leq H(\mu,p)+H(p,\nu).
\]
This completes our proof of Proposition~\ref{prop equiv IMS}.
\end{proof}

In view of the subadditivity of the function $f$ (see \eqref{assumpsublinearf}) and under the assumption \eqref{assumpcoercivef}, the functional $H$ is subadditive, and hence the $G$-transform is idempotent on $\mathcal{S}_{b}(\mathbb{R}^{d})$.
\begin{theorem} \label{cor: 1.20} Let $\mu, \nu \in \mathcal{P}_{2}(\mathbb{R}^{d})$ and \eqref{assumpcoercivef} hold. Then 	\begin{equation*}
	H(\mu,\nu)=\sup\{\langle \nu-\mu, \psi\rangle \st \psi \in \mathcal{U}_{b}(\mathbb{R}^{d}),\; -\psi=(-\psi)^{G}\}. \label{1.35}
	\end{equation*}
\end{theorem}
\begin{proof}
The proof follows from  Proposition~\ref{prop dual HI} as soon as $\varphi^{GG}=\varphi^{G}$ for each $\varphi \in \mathcal{S}_{b}(\mathbb{R}^{d})$. This, in view of  Proposition~\ref{prop equiv IMS}, holds  if for each $p \in \mathcal{P}_{2}(\mathbb{R}^{d})$, $\nu\in \mathcal{P}_{2}(\mathbb{R}^{d})$ and analytically measurable probability kernel $y \in \mathbb{R}^{d}\mapsto \gamma^{y} \in \mathcal{P}_{2}(\mathbb{R}^{d})$, 
\begin{equation}
G(x, \nu)\leq G(x,p)+\int_{\mathbb{R}^{d}}G(y,\gamma^{y})\diff p(y) \,\ \,\ \text{whenever} \,\ \nu=\int_{\mathbb{R}^{d}}\gamma^{y}\diff p(y). \label{1.31}
\end{equation}
Using the definition of $H$ (see \eqref{thedefofH}) and Proposition~\ref{prop lower bound}, we have
\begin{equation}\label{estdefH>=F}
\int_{\mathbb{R}^{d}}G(y,\gamma^{y}) \diff p(y) \geq H(p, \nu)\geq  F(p, \nu).
\end{equation}
Combining Proposition~\ref{prop: 1.16}~$\ref{item 2 lcs}$ and \eqref{estdefH>=F}, yields \eqref{1.31}, namely
\[
G(x,\nu)\leq G(x,p)+F(p,\nu)\leq G(x,p)+\int_{\mathbb{R}^{d}}G(y,\gamma^{y})\diff p(y)
\]
(see \eqref{1.16}). This completes our proof of Theorem~\ref{cor: 1.20}.
\end{proof}
\begin{rem}\label{rem 1.22}
	Let $E\subset \mathbb{R}^{d}$. Then in view of \eqref{1.16} and \eqref{G-transform ABC}, 
	\begin{equation*}
	\begin{split}
	&\;\ \;\ \;\ (-\psi)^{G}(x)= -\psi(x)\,\ \,\ \forall x \in E\\ &\Leftrightarrow \int_{\mathbb{R}^{d}}-\psi\diff p+G(x,p)\geq -\psi(x)\,\, \,\ \forall (x,p)\in E\times \mathcal{P}_{2}(\mathbb{R}^{d})\\
	&\Leftrightarrow\int_{\mathbb{R}^{d}}-\psi\diff p +\int_{\mathbb{R}^{d}}f\left(\frac{\diff\lambda}{\diff|\lambda|}\right)\diff|\lambda|\geq -\psi(x) \\ & \qquad\qquad\qquad \,\ \forall (x,p,\lambda) \in E\times \mathcal{P}_{2}(\mathbb{R}^{d})\times \mathcal{M}(\mathbb{R}^{d}, S^{+}_{d}),\,\  \tr\left(\frac{1}{2}\nabla^{2}\lambda\right)=p-\delta_{x}. 
	\end{split}
	\end{equation*}
\end{rem}
Assuming \eqref{assumpgrowthf} instead of \eqref{assumpcoercivef}, we also obtain the dual formulation for $H$, where the dual competitors are invariant under the $G$-transform but belong to $\Phi_{bb,2}(\mathbb{R}^{d})$ (see Theorem~\ref{thm dual for H through G-transform A4}). We first prove the following key result.
\begin{prop} \label{prop est by variance}
	Let \eqref{assumpgrowthf} hold. Then for each $p \in \mathcal{P}_{2}(\mathbb{R}^{d})$, 
\begin{equation}\label{est var <=ktr}
G([p],p)\leq \kappa_{1}\var(p).
\end{equation}
Furthermore, if $f=\tr$ on $S^{+}_{d}$, then for each $p \in \mathcal{P}_{2}(\mathbb{R}^{d})$, 
\begin{equation}\label{est var =ktr}
G([p],p)=\var(p).
\end{equation}
\end{prop}
\begin{proof} In view of Proposition~\ref{prop: 1.7} and \eqref{assumpgrowthf}, 
	\[
	G([p],p)\leq \kappa_{1}\inf\left\{\int_{\mathbb{R}^{d}}\tr\left(\frac{\diff \lambda}{\diff |\lambda|}\right) \diff |\lambda| \st \tr\left(\frac{1}{2} \nabla^{2}\lambda\right)=p- \delta_{[p]}\right\}.
	\]
	Thus, it suffices to prove \eqref{est var =ktr}. Assume that $f=\tr$ on $S^{+}_{d}$. 
	By Jensen's inequality, for each concave function $u:\mathbb{R}^{d}\to\mathbb{R}$,
	\begin{equation}\label{Jensen's bar ineq}
	\langle p-\delta_{[p]}, u\rangle \leq 0.
	\end{equation}
	Since for each $A\in S_{d}$, $-\tr^{*}(A)=0 \Leftrightarrow A-I_{d} \leq 0$ (see \eqref{symAtr}), according to Theorem~\ref{prop: 1.13},
	\begin{align*}
	G([p],p)&= \sup\left\{\langle p-\delta_{[p]}, \psi\rangle \st \psi \in C^{2}_{b}(\mathbb{R}^{d}),\; \left(\frac{1}{2}\nabla^{2}\psi-I_{d}\right)\leq 0 \,\ \text{on}\,\ \mathbb{R}^{d} \right\}\\
	& =\left \langle p-\delta_{[p]}, |\cdot|^{2} \right\rangle + \sup \left\{\langle p-\delta_{[p]},u\rangle \st u \in C^{2}_{b}(\mathbb{R}^{d}),\; u \,\ \text{is concave on} \,\ \mathbb{R}^{d}\right\}\\
	&= \var(p),
	\end{align*}
	where  the last equality comes from (\ref{Jensen's bar ineq}). This completes our proof of Proposition~\ref{prop est by variance}.
\end{proof}
\begin{cor}\label{cor semiconvex 1.22}
	Let \eqref{assumpgrowthf} hold. Then for each  lower semianalytic function  $\varphi:\mathbb{R}^{d}\to \mathbb{R}$ bounded from below such that $\varphi=\varphi^{G}$, the function $\varphi(\cdot) + \kappa_{1} |\cdot|^{2}$ is convex and locally Lipschitz on $\mathbb{R}^{d}$. 
\end{cor}
\begin{proof}[Proof of Corollary~\ref{cor semiconvex 1.22}]
	Let $\varphi:\mathbb{R}^{d}\to \mathbb{R}$ be lower semianalytic, bounded from below and $\varphi=\varphi^{G}$. According to Remark~\ref{rem 1.22}, for each $p \in  \mathcal{P}_{2}(\mathbb{R}^{d})$, 
	\begin{align*}
	\varphi([p]) \leq \int_{\mathbb{R}^{d}}\varphi \diff p + G([p],p) \leq \int_{\mathbb{R}^{d}}\varphi \diff p + \kappa_{1} \left(\int_{\mathbb{R}^{d}}|y|^{2}\diff p(y) - |[p]|^{2}\right),
	\end{align*}
	where the latter estimate comes from the estimate \eqref{est var <=ktr} of Proposition~\ref{prop est by variance}. 
	Thus,
	\begin{equation}\label{J ineq}
	\varphi([p]) + \kappa_{1}|[p]|^{2}\leq \int_{\mathbb{R}^{d}}\left(\varphi(y)+\kappa_{1}|y|^{2}\right)\diff p(y).
	\end{equation}
	Since 
	\[
	\frac{x+y}{2}=\left[\frac{\delta_{x}}{2}+\frac{\delta_{y}}{2}\right] \,\ \,\ \forall x,y \in \mathbb{R}^{d}, 
	\]
	(\ref{J ineq}) yields
	\[
	\varphi\left(\frac{x+y}{2}\right) + \kappa_{1}\left|\frac{x+y}{2}\right|^{2}\leq \frac{\varphi(x)+\varphi(y)}{2}+\frac{\kappa_{1}(|x|^{2}+|y|^{2})}{2},
	\]
	 which, together with the local boundedness of $\varphi$, implies the convexity of $\varphi(\cdot)+\kappa_{1}|\cdot|^{2}$ on $\mathbb{R}^{d}$. For the fact that a convex function is locally Lipschitz on the interior of its proper domain, the reader may consult \cite[Theorem~10.4]{Rockafellar}. This completes our proof of Corollary~\ref{cor semiconvex 1.22}.
\end{proof}
\begin{prop}\label{prop dual for H through Phi}
Let $\mu, \nu \in \mathcal{P}_{2}(\mathbb{R}^{d})$, \eqref{assumpgrowthf} hold and $\varphi^{GG}=\varphi^{G}$ for each $\varphi\in \Phi_{bb,2}(\mathbb{R}^{d})$. Then 
\begin{equation}\label{dual for H through CBF}
H(\mu, \nu) = \sup\{\langle \nu-\mu, \psi \rangle \st -\psi \in \Phi_{bb,2}(\mathbb{R}^{d}), \; -\psi=(-\psi)^{G}\}.
\end{equation}
\end{prop}
\begin{proof}
By Proposition~\ref{proposition 1.17}~$\ref{dualformh}$, $H(\mu,\nu)$ is greater than or equal to the supremum in \eqref{dual for H through CBF}. In view of Remark~\ref{remark G->phi is welld} and Corollary~\ref{cor semiconvex 1.22}, for each $\varphi \in \Phi_{bb,2}(\mathbb{R}^{d})$, the functions $\varphi^{G}$ and $\varphi^{GG}$ are well defined and $\varphi^{G} \in \Phi_{bb,2}(\mathbb{R}^{d})$  whenever $\varphi^{GG}=\varphi^{G}$. Then $\{\varphi^{G} \st \varphi \in \Phi_{bb,2}(\mathbb{R}^{d})\} \subset \{\varphi \st \varphi \in \Phi_{bb,2}(\mathbb{R}^{d}),\; \varphi=\varphi^{G}\}$, since $\varphi^{GG}=\varphi^{G}$ for each $\varphi \in \Phi_{bb,2}(\mathbb{R}^{d})$ by our assumption. Using this and the estimate $$\int_{\mathbb{R}^{d}}\varphi^{G}\diff \mu - \int_{\mathbb{R}^{d}}\varphi \diff \nu \leq \int_{\mathbb{R}^{d}}\varphi^{G}\diff \mu - \int_{\mathbb{R}^{d}}\varphi^{G} \diff \nu$$ for each $\varphi \in \Phi_{bb,2}(\mathbb{R}^{d})$ in \eqref{1.29}, we deduce that $H(\mu,\nu)$ is less than or equal to the supremum in \eqref{dual for H through CBF}. This completes our proof of Proposition~\ref{prop dual for H through Phi}.
\end{proof}
The next proposition is a counterpart of Proposition~\ref{prop equiv IMS}, where we replace the assumption \eqref{assumpcoercivef}  by \eqref{assumpgrowthf} and describe some situations in which the $G$-transform is idempotent on $\Phi_{bb,2}(\mathbb{R}^{d})$. 
\begin{prop}\label{prop equiv PhiMS}
Let \eqref{assumpgrowthf} hold. Then the following assertions are equivalent.
\begin{enumerate}[label=(\roman*)]
\item \label{Phib2 assertion I} For each choice of $\mu, \nu, p \in \mathcal{P}_{2}(\mathbb{R}^{d})$, $H(\mu, \nu)\leq H(\mu, p) + H(p, \nu)$.
\item \label{Phib2 assertion II} For each $\nu, p \in \mathcal{P}_{2}(\mathbb{R}^{d})$ and analytically measurable probability kernel $y \in \mathbb{R}^{d}\mapsto \gamma^{y}\in \mathcal{P}_{2}(\mathbb{R}^{d})$,
\begin{equation}
G(x,\nu) \leq G(x,p) + \int_{\mathbb{R}^{d}} G(y, \gamma^{y}) \diff p(y) \,\ \,\ \text{whenever}\,\ \nu=\int_{\mathbb{R}^{d}}\gamma^{y} \diff p(y).
\end{equation}
\item \label{Phib2 assertion III} For each $\varphi \in \Phi_{bb,2}(\mathbb{R}^{d})$, $\varphi^{GG}=\varphi^{G}$.
\end{enumerate}
\end{prop}
\begin{proof}
The proof follows by reproducing the arguments of the  proof of Proposition~\ref{prop equiv IMS} with minor modifications, in particular, using Proposition~\ref{prop dual for H through Phi} in the proof of the implication $\ref{Phib2 assertion III}\Rightarrow \ref{Phib2 assertion I}$.
\end{proof}

\begin{theorem}\label{thm dual for H through G-transform A4}
Let $\mu, \nu \in \mathcal{P}_{2}(\mathbb{R}^{d})$ and \eqref{assumpgrowthf} hold. Then
\[
H(\mu, \nu)=\sup\{\langle \nu-\mu, \psi \rangle \st -\psi \in \Phi_{bb,2}(\mathbb{R}^{d}),\; -\psi=(-\psi)^{G}\}.
\]
\end{theorem}
\begin{proof}
Proceeding in a similar manner to the proof of Theorem~\ref{cor: 1.20}, we deduce that the assertions $\ref{Phib2 assertion I}$-$\ref{Phib2 assertion III}$ of Proposition~\ref{prop equiv PhiMS} hold. Then, applying Proposition~\ref{prop dual for H through Phi}, we complete our proof of Theorem~\ref{thm dual for H through G-transform A4}.
\end{proof}
If  \eqref{assumpcoercivef} and \eqref{assumpgrowthf} hold simultaneously, we have the following dual formulation.
\begin{theorem}\label{thmdualHA1-A4}
Let $\mu, \nu \in \mathcal{P}_{2}(\mathbb{R}^{d})$ and \eqref{assumpcoercivef}, \eqref{assumpgrowthf} hold. Then
\begin{equation}\label{thmdualA3A4}
H(\mu, \nu)=\sup\{\langle \nu-\mu, \psi \rangle \st \psi \in C_{b}(\mathbb{R}^{d}), \; -\psi=(-\psi)^{G}\}.
\end{equation}
\end{theorem}
\begin{proof}
The proof follows from Theorem~\ref{cor: 1.20} and Corollary~\ref{cor semiconvex 1.22}.
\end{proof}
\subsection{Duality for $F$ in terms of invariant functions under $G$-transform}
The following theorem refines the dual formulation for $F(\mu, \nu)$ obtained in Theorem~\ref{prop: 1.13} in terms of invariant functions under $G$-transform $\varphi \mapsto \varphi^{G}$ (see \eqref{G-transform ABC}) in line with \cite{Alibert-Bouchitte-Champion-2019} and \cite{Gozlan_Roberto_Samson_Tetali}. 
\begin{theorem}  \label{prop: 1.36} For each $\mu, \, \nu \in \mathcal{P}_{2}(\mathbb{R}^{d})$, we have
	\begin{equation}\label{dualFG}
	F(\mu,\nu)=\sup\{\langle \nu - \mu, \psi \rangle \st \psi \in C^{2}_{b}(\mathbb{R}^{d}),\,\ -\psi= (-\psi)^{G}\}.
	\end{equation}
\end{theorem}
\begin{rem}\label{comment about dformulations for F and H}
The connection between the formulae \eqref{thmdualA3A4} and \eqref{dualFG} is striking, with the only difference being that the competitors for $F(\mu,\nu)$ are additionally twice continuously differentiable with respect to the competitors for $H(\mu, \nu)$. Moreover, Theorem~\ref{th F=G} indicates the equality of these formulae.
\end{rem}
The proof of Theorem~\ref{prop: 1.36} is a direct consequence of Theorem~\ref{prop: 1.13} and the next result.
\begin{prop} \label{lem: 1.21} Let $\psi \in C^{2}_{b}(\mathbb{R}^{d})$. Then $-f^{*}\left(\frac{1}{2}\nabla^{2}\psi\right)=0$ in $\mathbb{R}^{d}$ if and only if $-\psi=~(-\psi)^{G}$.\end{prop}
\begin{proof} According to Theorem~\ref{prop: 1.13}, for each  $x\in \mathbb{R}^{d}$, $p\in \mathcal{P}_{2}(\mathbb{R}^{d})$ and $\psi \in C^{2}_{b}(\mathbb{R}^{d})$ such that $-f^{*}(\frac{1}{2}\nabla^{2}\psi)=0$ in $\mathbb{R}^{d}$, 
	\begin{equation*}
	G(x,p)\geq \int_{\mathbb{R}^{d}}\psi \diff p - \psi(x)
	\end{equation*}
	and hence
	\begin{equation*}
	(-\psi)^{G}(x)=\inf\left\{\int_{\mathbb{R}^{d}}-\psi\diff p + G(x,p) \st p \in \mathcal{P}_{2}(\mathbb{R}^{d})\right\} \geq -\psi(x).
	\end{equation*}
	Thus, $-\psi=(-\psi)^{G}$ (see \eqref{eq ineq49050t5ti0i45t}).  
	
	Let us now assume that $\psi\in C^{2}_{b}(\mathbb{R}^{d})$ and $-\psi=(-\psi)^{G}$. Fix arbitrary $A=(a_{ij})_{i,j=1}^{d} \in S^{++}_{d}$ and $x_{0} \in \mathbb{R}^{d}$.  Let $g_{A}(x,y)$ be the Green function  of the elliptic operator $L=-\sum_{i,j=1}^{d}a_{ij}\partial _{ij}=-\textup{div}(A\nabla)$ on $B_{r}(x_{0})$, namely, for each $y \in B_{r}(x_{0})$,  $g_{A}(\cdot, y) \in W^{1,p}_{0}(B_{r}(x_{0}))$ whenever $p<d/(d-1)$ and 
	\begin{equation*}
		-A:\nabla^{2}g_{A}=\delta_{y} \,\ \,\ \text{in} \,\ \,\ \mathcal{D}^{\prime}(B_{r}(x_{0})),
	\end{equation*} 
	which means that
	\[
		-\int_{B_{r}(x_{0})}A:\nabla^{2}\varphi(x)g_{A}(x,y)\diff x = \varphi(y) \,\ \,\ \forall \varphi \in C^{2}_{c}(B_{r}(x_{0}))
	\]
	(see, for instance, \cite{Littman-Stampacchia-Weinberger-1963}). 	Since $A\in S^{++}_{d}$, $A=P\,\textup{diag}(\lambda_{1},\dotsc,\lambda_{d})\,P^{\mathrm{T}}$ for some orthogonal real $d\times d$ matrix $P$ and positive numbers $\lambda_{i}>0$. Also $A^{-1}\in S^{++}_{d}$ and there exists the unique matrix $B \in S^{++}_{d}$ such that $A=B^{-2}$, namely $B=P\,\textup{diag}(1/\sqrt{\lambda_{1}},\dotsc, 1/\sqrt{\lambda_{d}})\, P^{\mathrm{T}}$. Fix an arbitrary $v\in C^{2}_{b}(\mathbb{R}^{d})$ and define $u(\cdot)=v(B^{-1}\cdot)$ so that $u \in C^{2}_{b}(\mathbb{R}^{d})$. Then $A:\nabla^{2}v(x)=\Delta u (Bx)$ for each $x \in \mathbb{R}^{d}$. If $g$ is the Green function of the Laplace operator on  $BB_{r}(x_{0})=\{Bx \st  x \in B_{r}(x_{0})\}$, then $g_{A}(x,y)=\det(B) g(Bx, By)$. To lighten the notation, define $U=BB_{r}(x_{0})$. Next, using the Green representation formula, changing the variables and using that $A:\nabla^{2}v(x)=\Delta u (Bx)$ and $g_{A}(x,x_{0})=\det(B) g(Bx, Bx_{0})$,  we have 
	\begin{equation}\label{Green's A id}
	\begin{split}
	v(x_{0})&=u(Bx_{0})=-\int_{U}\Delta u(x)g(x, Bx_{0})\diff x - \int_{\partial U}u(x)\nabla g(x, Bx_{0})\cdot \nu_{\partial U}(x)\diff \mathcal{H}^{d-1}(x)\\
	&=-\det(B)\int_{B_{r}(x_{0})}\Delta u(Bx)g(Bx,Bx_{0})\diff x\\& \qquad  - \det(B)\int_{\partial B_{r}(x_{0})}u(Bx)\nabla g(Bx, Bx_{0})\cdot \nu_{\partial U}(Bx)|B^{-\mathrm{T}}\nu(x)|\diff \mathcal{H}^{d-1}(x)\\
	&=-\int_{B_{r}(x_{0})}A:\nabla^{2}v(x)g_{A}(x,x_{0})\diff x - \int_{\partial B_{r}(x_{0})}v(x)A\nabla g_{A}(x,x_{0})\cdot \nu(x)\diff \mathcal{H}^{d-1}(x),
	\end{split}
	\end{equation}
	where $\nu_{\partial U}$ and $\nu$ denote the outward pointing unit normal vector fields along $\partial U$ and $\partial B_{r}(x_{0})$, respectively. 
		
	By Hopf's lemma (see \cite{Hopff, Oleinik-1952}),  $-A\nabla g_{A}(x, x_{0})\cdot \nu(x)>0$ for each $x \in \partial B_{r}(x_{0})$. Then, using \eqref{Green's A id} with  $v=1$, we obtain $p=-A\nabla g_{A}(\cdot, x_{0})\cdot \nu(\cdot) \mathcal{H}^{d-1}\mres \partial B_{r}(x_{0})\in \mathcal{P}_{2}(\mathbb{R}^{d})$. 
	
	 Altogether, due to \eqref{Green's A id} and the fact that $v\in C^{2}_{b}(\mathbb{R}^{d})$ was arbitrarily chosen, we have 
	\begin{equation}\label{eq 1m43i9jt53jtj30jk0}
	\textup{tr}\left(\frac{1}{2}\nabla^{2}\lambda\right) = p - \delta_{x_{0}},
	\end{equation}
	where $ \lambda=2Ag_{A}(\cdot,x_{0}) \mathcal{L}^{d} \mres B_{r}(x_{0}) \in \mathcal{M}(\mathbb{R}^{d}, S^{+}_{d})$ and $|\lambda|(\mathbb{R}^{d})<+\infty$. Since $-\psi(x_{0})=(-\psi)^{G}(x_{0})$, according to Remark~\ref{rem 1.22}, 
		\[
	\int_{\mathbb{R}^{d}}f\left(\frac{\diff \lambda}{\diff |\lambda|}\right)\diff |\lambda| \geq \int_{\mathbb{R}^{d}}\psi \diff p - \psi(x_{0}).
	\]
	This, together with the positive  1-homogeneity of $f$ (see \eqref{assumpsublinearf}), \eqref{eq 1m43i9jt53jtj30jk0} and the fact that $\psi\in C^{2}_{b}(\mathbb{R}^{d})$, implies that
	\begin{equation}\label{=> dom f *}
	\int_{B_{r}(x_{0})}f(A)g_{A}(x,x_{0})\diff x \geq \int_{B_{r}(x_{0})}A:\frac{1}{2}\nabla^{2}\psi(x)g_{A}(x,x_{0})\diff x.
	\end{equation}
	Assume by contradiction that $A:\frac{1}{2}\nabla^{2}\psi(x_{0})-f(A)>0$. Then there exist $\varepsilon, r >0$ such that for each $x \in B_{r}(x_{0})$, $A:\frac{1}{2}\nabla^{2}\psi(x)-f(A)\geq \varepsilon$. But this contradicts (\ref{=> dom f *}), since $\int_{B_{r}(x_{0})}g_{A}(x,x_{0})\diff x>0$. Thus,  $A:\frac{1}{2}\nabla^{2}\psi(x_{0})-f(A) \leq 0$, which implies that $\frac{1}{2}\nabla^{2}\psi(x_{0})\in \dom(f^{*})$, because $A\in S^{++}_{d}$ was arbitrarily chosen and $S^{++}_{d}\cap \dom(f)$ is dense in $\dom(f)$ (see \eqref{assumpdom}). Since $x_{0} \in \mathbb{R}^{d}$ was arbitrary, $\frac{1}{2}\nabla^{2}\psi(\mathbb{R}^{d})\subset \dom(f^{*})$, which holds if and only if $-f^{*}\left(\frac{1}{2}\nabla^{2} \psi\right)=0$ in $\mathbb{R}^{d}$. This completes our proof of Proposition~\ref{lem: 1.21}.
\end{proof}
\section{$F$ versus $H$} \label{Section 4}
\subsection{$F=H$ under approximation assumptions}
Our first type of approximation assumption is related to the assumption \eqref{assumpcoercivef}.
\begin{theorem}\label{thmsmooth3}
Let $\mu, \nu \in \mathcal{P}_{2}(\mathbb{R}^{d})$ and \eqref{assumpcoercivef} hold. Assume that for each $\psi \in \mathcal{U}_{b}(\mathbb{R}^{d})$ such that $-\psi=(-\psi)^{G}$ there exists $(\psi_{n})_{n\in \mathbb{N}} \subset C^{2}_{b}(\mathbb{R}^{d})$ such that $-\psi_{n}=(-\psi_{n})^{G}$ for each $n\in \mathbb{N}$ and $\langle \nu-\mu, \psi_{n}\rangle \to \langle \nu-\mu, \psi\rangle$ as $n\to+\infty$. Then $F(\mu, \nu)=H(\mu,\nu)$.
\end{theorem}
\begin{proof}
By Proposition~\ref{prop lower bound}, $F(\mu, \nu)\leq H(\mu, \nu)$. Next, using the assumption of Theorem~\ref{thmsmooth3}, together with Theorem~\ref{cor: 1.20} and Theorem~\ref{prop: 1.36}, we have $H(\mu, \nu) \le F(\mu, \nu)$, which completes our proof of Theorem~\ref{thmsmooth3}.
\end{proof}
\noindent \textbf{Example}~\ref{thirdexample}.
For each $A\in S_{d}$, 
	\begin{equation}\label{f=tId}
	f(A)=\begin{cases}
	t \,\ &\text{if}\,\ A=tI_{d} \,\ \text{for some}\,\ t \geq 0,\\
	+\infty \,\ &\text{otherwise}.
	\end{cases}
	\end{equation}
	Clearly, $f$ satisfies \eqref{assumpdom}-\eqref{assumpcoercivef}. Then for each $A\in S_{d}$, 
	\begin{equation}\label{Tfenchel1}
	f^{*}(A)=\sup\{A:M-f(M) \st M\in \dom(f)\}
	=\begin{cases} 
	0 \,\ &\text{if}\,\ \tr(A)\leq 1,\\
	+\infty \,\ &\text{otherwise}.
	\end{cases}	
	\end{equation}
	Given $\mu, \nu \in \mathcal{P}_{2}(\mathbb{R}^{d})$, using Theorem~\ref{prop: 1.13} and \eqref{Tfenchel1}, we have
	\begin{align*}
	F(\mu,\nu)&=\sup\left\{\langle \nu - \mu, \psi \rangle \st \psi \in C^{2}_{b}(\mathbb{R}^{d}),\; \tr\left(\frac{1}{2}\nabla^{2}\psi(x)\right)\leq 1 \,\ \forall x \in \mathbb{R}^{d} \right\}\\
	&=\sup\left\{\langle \nu-\mu, \psi \rangle \st \psi \in C^{2}_{b}(\mathbb{R}^{d}),\; \Delta\left[\frac{1}{2}\psi(x) -\frac{|x|^{2}}{2d}\right] \leq 0 \,\ \forall x \in \mathbb{R}^{d}\right\}\\
	&=\left\langle \nu - \mu, \frac{1}{d} |\cdot|^{2} \right \rangle + \sup\{\langle \nu- \mu, \varphi\rangle \st \varphi \in C^{2}_{b}(\mathbb{R}^{d}), \; \Delta \varphi(x) \le 0 \,\ \forall x \in \mathbb{R}^{d}\}\\
	&=\begin{cases}\numberthis \label{F=H f=tid 1}
	 \displaystyle \frac{1}{d}\var(\nu)- \frac{1}{d}\var(\mu) \,\ &\text{if}\,\ \mu \leq_{sh}\nu, \\
	+\infty \,\ &\text{otherwise}
	\end{cases}\\
	&=\begin{cases}\numberthis \label{explftrace}
\displaystyle f\left(\int_{\mathbb{R}^{d}} x\otimes x \diff (\nu(x)- \mu(x))\right) &\text{if}\,\ \mu\leq_{sh}\nu,\\
+\infty \,\ &\text{otherwise},
\end{cases}
	\end{align*}
	where to obtain \eqref{F=H f=tid 1} we have used the fact that if $\mu\leq_{sh}\nu$ (see Definition~\ref{subharmonic order measures}), then  $[\mu]=[\nu]$ and hence $\left\langle \nu- \mu, \frac{1}{d}|\cdot|^{2}\right\rangle = \frac{1}{d}\var(\nu)-\frac{1}{d} \var(\mu)$. 
		To obtain \eqref{explftrace} we have used the following. Notice that a competitor $\lambda$ for (\ref{1.10}) such that $\int_{\mathbb{R}^{d}}f(\frac{\diff \lambda}{\diff|\lambda|})\diff|\lambda|<+\infty$ with $f$ defined by \eqref{f=tId} has the form $\lambda=u\,I_{d}\,m$, where $m$ is a nonnegative measure on $\mathbb{R}^{d}$, $u\geq 0$ $m$-a.e. on $\mathbb{R}^{d}$ and $u \in L^{1}(\mathbb{R}^{d}, \diff m)$. Furthermore, such a measure $\lambda$ solves $\tr\left(\frac{1}{2}\nabla^{2}\lambda\right) = \nu- \mu$ if and only if $\Delta(\frac{1}{2}u m)=\nu - \mu$ in the weak sense (see (\ref{1.9})). Thus, for each $i,j \in \{1,\dotsc,d\}$ such that $i\neq j$, we have
\begin{equation}\label{test x^2_i-x^2_j}
\int_{\mathbb{R}^{d}}(x^{2}_{i}-x^{2}_{j})\diff \nu(x) -\int_{\mathbb{R}^{d}}(x^{2}_{i}-x^{2}_{j})\diff \mu(x)=\int_{\mathbb{R}^{d}}\frac{1}{2}\Delta[x^{2}_{i}-x^{2}_{j}]u(x)\diff m(x)=0
\end{equation}
and
\begin{equation}\label{test x_ix_j}
\int_{\mathbb{R}^{d}}x_{i}x_{j}\diff \nu(x)-\int_{\mathbb{R}^{d}}x_{i}x_{j}\diff \mu(x)= \int_{\mathbb{R}^{d}}\frac{1}{2}\Delta[x_{i}x_{j}]u(x)\diff m(x)=0.
\end{equation}
If $\mu\le_{sh} \nu$, by \eqref{F=H f=tid 1}, $F(\mu,\nu)<+\infty$.  Then, (\ref{test x^2_i-x^2_j}) and (\ref{test x_ix_j}) imply that $\int_{\mathbb{R}^{d}} x\otimes x \diff (\nu(x)-\mu(x)) \in S^{+}_{d}$ is a diagonal matrix equal to $\int_{\mathbb{R}^{d}}\frac{1}{d}|x|^{2}\diff(\nu(x)-\mu(x))I_{d}$ and hence
\begin{equation*}
f\left(\int_{\mathbb{R}^{d}} x\otimes x \diff (\nu(x)-\mu(x))\right)=\int_{\mathbb{R}^{d}}\frac{1}{d}|x|^{2}\diff(\nu(x)-\mu(x))=\frac{1}{d}\var(\nu)-\frac{1}{d}\var(\mu).
\end{equation*}
Thus, for each $a \in \mathbb{R}^{d}$, for each $\xi \in \mathbb{R}^{d}$ and for each $A\in S_{d}$ such that $\tr(A)=0$, the function $$\psi(x)=a+ x \cdot \xi + \left(A+\frac{I_{d}}{d} \right):x\otimes x $$ is a dual optimizer for $F(\mu,\nu)$, which is understood in the setting of \eqref{Dual}, when $F(\mu, \nu)<+\infty$ and $f$ is defined by \eqref{f=tId}.

Notice that the equality $F(\mu,\nu)=H(\mu,\nu)$ when $f$ is defined by \eqref{f=tId}  can be proved based on the results from \cite{Ghoussoub-Kim-Lin}, as well as using Theorem~\ref{thmsmooth3}. Below we present both approaches allowing the reader to verify the assumptions of Theorem~\ref{thmsmooth3} in the context of the example~\ref{thirdexample}.
\begin{itemize} 
\item By \eqref{F=H f=tid 1}, if $F(\mu, \nu) <+\infty$, then $\mu$ and $\nu$ are in subharmonic order. It follows by Theorem~1.5 and Remark~1.7 from \cite{Ghoussoub-Kim-Lin} that in this case  there exists at least one transport plan $\gamma \in \Pi(\mu, \nu)$ such that $\gamma=\gamma^{x}\otimes \mu$ with $\delta_{x}\leq_{sh}\gamma^{x}$ for $\mu$-a.e. $x \in \mathbb{R}^{d}$. For this transport, one gets
\begin{align*}
\int_{\mathbb{R}^{d}}G(x,\gamma^{x}) \diff \mu(x) &= \int_{\mathbb{R}^{d}}F(\delta_{x}, \gamma^{x}) \diff \mu(x)\\
&=\int_{\mathbb{R}^{d}}\frac{1}{d}(\mathrm{var}(\gamma^{x})-\mathrm{var}(\delta_{x})) \diff \mu(x)=\frac{1}{d}\int_{\mathbb{R}^{d}}\mathrm{var}(\gamma^{x})\diff \mu(x)\\
&=\frac{1}{d}\int_{\mathbb{R}^{d}}\left(\int_{\mathbb{R}^{d}}|y|^{2}\diff \gamma^{x}(y)- \left|\int_{\mathbb{R}^{d}}y\diff \gamma^{x}(y)\right|^{2}\right) \diff \mu(x)\\
&=\frac{1}{d} \left(\int_{\mathbb{R}^{d}}|y|^{2} \diff \nu(y) - \int_{\mathbb{R}^{d}}|x|^{2} \diff \mu(x)\right) = F(\mu,\nu),
\end{align*}
where we have used the facts that $x=[\gamma^{x}]$ and $[\mu]=[\nu]$. The equality $F(\mu,\nu)=H(\mu,\nu)$ then directly follows from Proposition~\ref{prop lower bound}.
\item We check the assumptions of Theorem~\ref{thmsmooth3}. For each  $r>0$ and $y \in \mathbb{R}^{d}$, we define the measure
 $\lambda_{y,r}=2I_{d}\,g(\cdot, y) \mathcal{L}^{d}\mres B_{r}(y) \in \mathcal{M}(\mathbb{R}^{d}, S^{+}_{d})$, where $g$ is the Green function of the Laplacian on $B_{r}(y)$. Let $\psi \in \mathcal{U}_{b}(\mathbb{R}^{d})$ and $-\psi=(-\psi)^{G}$. Using the facts that \[\tr\left(\frac{1}{2}\nabla^{2}\lambda_{y,r}\right)=\mathcal{H}^{d-1}(\partial B_{r}(y))^{-1}\mathcal{H}^{d-1}\mres \partial B_{r}(y) - \delta_{y}\] and
\[
\int_{\mathbb{R}^{d}}f\left(\frac{\diff \lambda_{y,r}}{\diff |\lambda_{y,r}|}\right)\diff |\lambda_{y,r}|=\int_{B_{r}(y)}2g(x,y)\diff x =\frac{r^{2}}{d}, 
\]
according to Remark~\ref{rem 1.22}, we have
\[
\fint_{\partial B_{r}(y)}-\psi(x) \diff \mathcal{H}^{d-1}(x) + \frac{r^{2}}{d} \geq -\psi(y).
\]
This implies that
\begin{equation*}\label{203i94jfgitngt5h60j6m}
\fint_{\partial B_{r}(y)}\left(-\psi(x)+\frac{|x|^{2}}{d}\right)\diff \mathcal{H}^{d-1}(x) \geq -\psi(y) + \frac{|y|^{2}}{d}.
\end{equation*}
Then, by Definition~\ref{defofsubharmonic}, $\Psi(\cdot):=-\psi(\cdot)+\frac{1}{d}|\cdot|^{2} \in \mathcal{SH}(\mathbb{R}^{d})$.  Fix an arbitrary $\varepsilon>0$.  Defining for each $x \in \mathbb{R}^{d}$, 
\[
\Psi_{\varepsilon}(x)=\int_{\mathbb{R}^{d}}\Psi(y)\eta_{\varepsilon}(x-y)\diff y,
 \]
where $\eta_{\varepsilon}(\cdot)=\varepsilon^{-d}\eta(\cdot/\varepsilon)\in C^{\infty}_{c}(\mathbb{R}^{d})$ and $\eta$ is a standard mollifier as in Lemma~\ref{lem 1.36},
we observe that $\Psi_{\varepsilon} \in C^{2}_{b}(\mathbb{R}^{d})$ is subharmonic on $\mathbb{R}^{d}$ (the reader may consult the proof of Lemma~\ref{lem 1.36}). For each $p\in \mathcal{P}_{2}(\mathbb{R}^{d})$ and $y\in \mathbb{R}^{d}$ such that $\Delta(\frac{1}{2}u\, m)=p-\delta_{y}$ for some nonnegative measure $m$ on $\mathbb{R}^{d}$ and $u\in L^{1}(\mathbb{R}^{d}, \diff m)$ such that $u\geq 0$ $m$-a.e. on $\mathbb{R}^{d}$, it holds $\delta_{y}\le_{sh} p$. Thus,
\begin{equation}\label{eqnmfv04j05j4gj5gu4jg54jo}
\int_{\mathbb{R}^{d}}\Psi_{\varepsilon}(x)\diff p(x)\geq \Psi_{\varepsilon}(y),
\end{equation}
since $\Psi_{\varepsilon}\in C^{2}_{b}(\mathbb{R}^{d})\cap \mathcal{SH}(\mathbb{R}^{d})$ (see Definition~\ref{subharmonic order measures}). Next, observing that $\frac{1}{d} |\cdot|^{2}_{\varepsilon} \in C^{2}_{b}(\mathbb{R}^{d}) \cap \mathcal{SH}(\mathbb{R}^{d})$, where 
\[
\frac{|x|^{2}_{\varepsilon}}{d}=\int_{\mathbb{R}^{d}}\frac{|y|^{2}}{d}\eta_{\varepsilon}(x-y)\diff y \,\ \text{and}\,\ \Delta \frac{|x|^{2}_{\varepsilon}}{d}=2 \,\ \,\ \forall x \in \mathbb{R}^{d},
\]
if $\Delta(\frac{1}{2}u m)=p-\delta_{y}$, we have
\[
\int_{\mathbb{R}^{d}}u\diff m=\int_{\mathbb{R}^{d}}\frac{1}{2}\Delta \left(\frac{|x|^{2}_{\varepsilon}}{d}\right)u(x)\diff m(x)=\int_{\mathbb{R}^{d}}\frac{|x|^{2}_{\varepsilon}}{d}\diff p(x) - \frac{|y|^{2}_{\varepsilon}}{d}.
\]
Using \eqref{eqnmfv04j05j4gj5gu4jg54jo}, the above formula and the fact that 
\[
\int_{\mathbb{R}^{d}}u\diff m  = \int_{\mathbb{R}^{d}}f\left(\frac{\diff \lambda}{\diff |\lambda|}\right) \diff |\lambda|,
\]
where $\diff \lambda=uI_{d}\diff m$, we get
\[
\int_{\mathbb{R}^{d}}\left(\Psi_{\varepsilon}(x)-\frac{|x|^{2}_{\varepsilon}}{d}\right)\diff p(x)+\int_{\mathbb{R}^{d}}f\left(\frac{\diff \lambda}{\diff |\lambda|}\right)\diff |\lambda| \geq \Psi_{\varepsilon}(y)-\frac{|y|^{2}_{\varepsilon}}{d}.
\]
This, in view of Remark~\ref{rem 1.22}, implies that $\Psi_{\varepsilon} -\frac{1}{d} |\cdot|^{2}_{\varepsilon} =\left(\Psi_{\varepsilon}-\frac{1}{d} |\cdot|^{2}_{\varepsilon} \right)^{G}$ on $\mathbb{R}^{d}$. Fix now a sequence of sufficiently small positive numbers $(\varepsilon_{n})_{n\in \mathbb{N}}$ such that $\varepsilon_{n}\to 0+$ as $n\to +\infty$. Observe that 
\[
\Psi_{\varepsilon_{n}}(x)-\frac{|x|^{2}_{\varepsilon_{n}}}{d} \to -\psi(x) \,\ \,\ \forall x \in \mathbb{R}^{d}
\]
(here we use that the convolution of a subharmonic function converges to this function everywhere, in view of the monotonicity condition of subharmonic functions; see, for instance, \cite[Section~2.9]{Rado-1937}). Altogether, we have defined the approximation sequence $(\psi_{n})_{n\in \mathbb{N}}=(-\Psi_{\varepsilon_{n}}+\frac{1}{d}|\cdot|^{2}_{\varepsilon_{n}})_{n\in \mathbb{N}}$ for $\psi$ in the sense of Theorem~\ref{thmsmooth3}. Therefore, according to Theorem~\ref{thmsmooth3}, $F=H$ when $f$ is defined by \eqref{f=tId}.
\end{itemize}
\begin{rem}\label{rem existence of Brownian martingale}
Applying Theorem~\ref{thmsmooth3} when $f$ is defined by \eqref{f=tId} allows us to recover the result of \cite{Ghoussoub-Kim-Lin}, namely the fact that whenever $\mu$ and $\nu$ are in subharmonic order, then there exists at least one transport $\gamma \in \Pi(\mu,\nu)$ such that $\gamma=\gamma^{x}\otimes \mu$ with $\delta_{x}\leq_{sh} \gamma^{x}$ for $\mu$-a.e. $x\in \mathbb{R}^{d}$. 
\end{rem}
Our second type of approximation assumption is related to the assumption \eqref{assumpgrowthf}.
\begin{theorem}\label{thmsmooth4}
Let $\mu, \nu \in \mathcal{P}_{2}(\mathbb{R}^{d})$ and \eqref{assumpgrowthf} hold. Assume that for each $\varphi \in \Phi_{bb,2}(\mathbb{R}^{d})$ such that $\varphi=\varphi^{G}$ there exists a sequence $(\varphi_{n})_{n\in \mathbb{N}} \subset C^{2}_{b}(\mathbb{R}^{d})$ such that $\varphi_{n}=\varphi^{G}_{n}$ for each $n\in \mathbb{N}$ and $\langle \nu-\mu, \varphi_{n}\rangle \to \langle \nu-\mu, \varphi \rangle$ as $n \to +\infty$. Then  $F(\mu, \nu)=H(\mu,\nu)$.
\end{theorem}
\begin{proof}
By Proposition~\ref{prop lower bound}, $F(\mu, \nu)\leq H(\mu, \nu)$. Next, using the assumption of Theorem~\ref{thmsmooth4}, together with Theorem~\ref{thm dual for H through G-transform A4} and Theorem~\ref{prop: 1.36}, we have $H(\mu, \nu) \le F(\mu, \nu)$, which completes our proof of Theorem~\ref{thmsmooth4}.
\end{proof}
\noindent \textbf{Example}~\ref{notcoerc}. For some $B \in S^{+}_{d}$ and for each $A\in S_{d}$,
\begin{equation}\label{defoffinex4}
f(A)=\begin{cases}
A:B \,\ &\text{if} \,\ A \in S^{+}_{d},\\
+\infty \,\ &\text{otherwise}.
\end{cases}
\end{equation}
Then $f$ satisfies \eqref{assumpdom}-\eqref{assumplscf}, \eqref{assumpgrowthf} and for each $A \in S_{d}$,
\begin{equation}\label{calc f* 4}
f^{*}(A)= \sup\{(A-B):M \st M \in S^{+}_{d}\}=\begin{cases}
0 \,\ & \text{if} \,\ A-B \leq 0, \\
+\infty \,\ & \text{otherwise}.
\end{cases}
\end{equation}
Given $\mu, \nu \in \mathcal{P}_{2}(\mathbb{R}^{d})$, using Theorem~\ref{prop: 1.13} and \eqref{calc f* 4}, we obtain
\begin{align*}
F(\mu, \nu)&=\sup\left\{\langle \nu-\mu, \psi \rangle \st \psi \in C^{2}_{b}(\mathbb{R}^{d}),\; \frac{1}{2} \nabla^{2}\psi(x) -B \leq 0 \,\ \forall x \in \mathbb{R}^{d}\right\}\\
&=\int_{\mathbb{R}^{d}} B: x\otimes x \diff (\nu(x)- \mu(x)) + \sup \left\{\langle \nu -\mu, \varphi \rangle \st \varphi \in C^{2}_{b}(\mathbb{R}^{d}), \; \varphi \; \text{is concave on} \; \mathbb{R}^{d}\right\}\\
&=\begin{cases}\numberthis \label{calc F ex4}
\displaystyle \int_{\mathbb{R}^{d}}B:x\otimes x \diff (\nu(x)-\mu(x)) \,\ &\text{if}\,\ \mu\le_{c}\nu,\\
+\infty \,\ &\text{otherwise}
\end{cases}\\
&=\begin{cases}
\displaystyle f\left(\int_{\mathbb{R}^{d}} x\otimes x \diff (\nu(x)-\mu(x))\right) \,\ & \text{if} \,\ \mu\le_{c}\nu, \\
+\infty \,\ & \text{otherwise},
\end{cases}
\end{align*}
where to obtain \eqref{calc F ex4} we have used Remark~\ref{remarkaboutconvexorder}. Notice that for each $a \in \mathbb{R}$ and for each $\zeta \in \mathbb{R}^{d}$, the function
\[
\psi(x)=a+x\cdot \zeta + B:x\otimes x
\]
is a dual optimizer for $F(\mu,\nu)$, which is understood in the setting of \eqref{Dual}, when $F(\mu,\nu)<+\infty$ and $f$ is defined by \eqref{defoffinex4}.

The equality $F(\mu,\nu)=H(\mu,\nu)$ when $f$ is defined by \eqref{defoffinex4} can be proved based on the Strassen theorem, as well as using Theorem~\ref{thmsmooth4}. We present both approaches, which allows the reader to verify the assumptions of Theorem~\ref{thmsmooth4} in the context of the example~\ref{notcoerc}.
\begin{itemize}
\item By \eqref{calc F ex4}, if $F(\mu,\nu)<+\infty$, then $\mu\leq_{c}\nu$. It follows by the Strassen theorem (see \cite{Strassen-1965}) that there exists $\gamma \in \Pi(\mu,\nu)$ such that $\delta_{x} \leq_{c} \gamma^{x}$ for $\mu$-a.e. $x \in \mathbb{R}^{d}$ for which
\begin{align*}
\int_{\mathbb{R}^{d}} G(x, \gamma^{x}) \diff \mu(x) &= \int_{\mathbb{R}^{d}} F(\delta_{x}, \gamma^{x})\diff \mu(x) \\ & =\int_{\mathbb{R}^{d}}\left(\int_{\mathbb{R}^{d}} B: y\otimes y \diff \gamma^{x}(y) - B: x \otimes x \right) \diff \mu(x)  \\  & = F(\mu,\nu).
\end{align*}
The equality $F(\mu,\nu)=H(\mu,\nu)$ then directly follows from Proposition~\ref{prop lower bound}.
\item We check the assumptions of Theorem~\ref{thmsmooth4}. Let $\varphi \in \Phi_{2}(\mathbb{R}^{d})$. We claim that $\varphi=\varphi^{G}$ if and only if the function $\varphi_{B} \in \Phi_{2}(\mathbb{R}^{d})$ defined by $\varphi_{B}(x)=\varphi(x)+B:x\otimes x$ is convex on $\mathbb{R}^{d}$. Indeed, if $\varphi_{B}$ is convex,  using Jensen's inequality and \eqref{calc F ex4}, for each $p \in \mathcal{P}_{2}(\mathbb{R}^{d})$, we have 
\begin{equation}\label{estimbothsides7}
\varphi([p])\le \int_{\mathbb{R}^{d}}(\varphi(y)+B:y\otimes y) \diff p(y)  - B: [p] \otimes [p]=\int_{\mathbb{R}^{d}} \varphi \diff p + G([p],p),
\end{equation}
which, in view of Remark~\ref{rem 1.22} and \eqref{calc F ex4}, yields $\varphi^{G}=\varphi$ on $\mathbb{R}^{d}$. The same argument yields that if $\varphi=\varphi^{G}$ on $\mathbb{R}^{d}$, then for each $p \in \mathcal{P}_{2}(\mathbb{R}^{d})$, \eqref{estimbothsides7} holds 
and hence
\begin{equation}\label{jinex4}
\varphi_{B}([p])\le \int_{\mathbb{R}^{d}} \varphi_{B} \diff p.
\end{equation}
For each $x, y \in \mathbb{R}^{d}$, choosing $p=\frac{1}{2}(\delta_{x}+\delta_{y})$ in \eqref{jinex4}, we obtain 
\[
\varphi_{B}\left(\frac{x+y}{2}\right)\le \frac{\varphi_{B}(x)}{2}+\frac{\varphi_{B}(y)}{2},
\]
which, since $\varphi_{B}$ is continuous on $\mathbb{R}^{d}$, implies that $\varphi_{B}$ is convex on $\mathbb{R}^{d}$. This completes the proof of our claim. 

Let $\varphi \in \Phi_{bb,2}(\mathbb{R}^{d})$ satisfy $\varphi=\varphi^{G}$. Then $\varphi_{B}$ is convex on $\mathbb{R}^{d}$. Furthermore, there exists a sequence $(\psi_{k})_{k\in \mathbb{N}} \subset C^{2}_{b}(\mathbb{R}^{d})$ of convex functions such that $\int_{\mathbb{R}^{d}} \psi_{k} \diff \mu \to \int_{\mathbb{R}^{d}} \varphi_{B} \diff \mu$ and $\int_{\mathbb{R}^{d}} \psi_{k} \diff \nu \to \int_{\mathbb{R}^{d}} \varphi_{B} \diff \nu$ as $k \to +\infty$ (we refer to Remark~\ref{remarkaboutconvexorder}). For each $k \in \mathbb{N}$, define $\varphi_{k}(x)=\psi_{k}(x)-B:x\otimes x$ for each $x \in \mathbb{R}^{d}$. Since $\psi_{k} \in C^{2}_{b}(\mathbb{R}^{d})$ is convex, $\varphi_{k}=\varphi_{k}^{G} \in C^{2}_{b}(\mathbb{R}^{d})$. This defines the approximation sequence for $\varphi$ in the sense of Theorem~\ref{thmsmooth4}, since \[\int_{\mathbb{R}^{d}} \varphi_{k} \diff \mu \to \int_{\mathbb{R}^{d}} (\varphi_{B}(x) -B:x\otimes x)\diff \mu(x)= \int_{\mathbb{R}^{d}} \varphi \diff \mu\] and \[\int_{\mathbb{R}^{d}} \varphi_{k} \diff \nu \to \int_{\mathbb{R}^{d}}(\varphi_{B}(x)-B:x\otimes x)\diff \nu(x)=\int_{\mathbb{R}^{d}} \varphi \diff \nu\] as $k \to +\infty$. Therefore, according to Theorem~\ref{thmsmooth4}, $F=H$ when $f$ is defined by \eqref{defoffinex4}.
\end{itemize}
\begin{rem}\label{rem Strassen theorem}
Applying Theorem~\ref{thmsmooth4} when $f$ is defined by \eqref{defoffinex4} allows us to recover the Strassen theorem, namely the fact that whenever $\mu$ and $\nu$ are in convex order, there exists at least one transport $\gamma \in \Pi(\mu,\nu)$ such that $\gamma=\gamma^{x}\otimes \mu$ with $\delta_{x}\leq_{c} \gamma^{x}$ for $\mu$-a.e. $x\in \mathbb{R}^{d}$. The same observation holds in the context of Theorem~\ref{th F=G} when $f$ is defined by \eqref{exoffbytrace}.
\end{rem}
\subsection{Viscosity solutions}
In this subsection, under the assumption \eqref{assumpcoercivef}, we characterize the functions $\psi \in C_{b}(\mathbb{R}^{d})$ such that $-\psi=(-\psi)^{G}$ as viscosity solutions of the Hamilton-Jacobi-Bellman equation $-f^{*}(\frac{1}{2}\nabla^{2}u)=0$ in $\mathbb{R}^{d}$ (notice that $f^{*}$  is discontinuous and takes its values in $\{0, +\infty\}$, see \eqref{mf f usual}). 
\begin{rem} Let \eqref{assumpcoercivef} hold. Since $\dom(f)\cap S^{++}_{d}$ is dense in $\dom(f)\not = \emptyset$ (see \eqref{assumpdom}), for each $A\in S_{d}$, 
\begin{equation}\label{Tf1}
f^{*}(A)
=\sup_{t>0}t(\mathcal{T}(A)-1)
=\begin{cases}
0 \,\ &\text{if} \,\ \,\  \mathcal{T}(A)\leq 1,\\
+\infty \,\ &\text{otherwise},
\end{cases}
\end{equation}
where
\begin{equation}\label{Trace by f}
\mathcal{T}(A)=\sup\{A:E \st E \in S^{++}_{d},\; f(E)=1\}.
\end{equation}
\end{rem}
In view of Proposition~\ref{lem: 1.21} and \eqref{Tf1}, if $\psi \in C^{2}_{b}(\mathbb{R}^{d})$, then $-\psi=(-\psi)^{G}$ on $\mathbb{R}^{d}$ if and only if $1-\mathcal{T}\left(\frac{1}{2} \nabla^{2} \psi\right) \geq 0$ in $\mathbb{R}^{d}$. However, a function $\psi \in C_{b}(\mathbb{R}^{d})$ such that $-\psi=(-\psi)^{G}$ may not be regular enough to define $\nabla^{2} \psi$ in the classical sense. Using the theory of viscosity solutions, we can define $\nabla^{2} \psi$ in the viscosity sense. Taking into account Theorems~\ref{thmdualHA1-A4} and \ref{prop: 1.36}, to derive the equality between $F$ and $H$, we first show that each  $\psi \in C_{b}(\mathbb{R}^{d})$ such that $-\psi=(-\psi)^{G}$ is a viscosity supersolution of the equation $1-\mathcal{T}\left(\frac{1}{2}\nabla^{2} u\right)= 0$ in $\mathbb{R}^{d}$ (see Definition~\ref{def of subsolution}).

    Following \cite{Crandall-Ishii-Lions-1992}, we shall say that a function $\mathscr{F} :S_{d} \to \mathbb{R}$ is \textit{proper} (also called \textit{degenerate elliptic}) if 
    \begin{equation}\label{proper}
     \mathscr{F}(A_{2})\leq \mathscr{F}(A_{1}) \,\ \text{whenever}\,\ A_{1}\leq A_{2}.
    \end{equation}

\begin{rem}
For each $c\in \mathbb{R}$ and $A \in S^{++}_{d}$, the function $\mathscr{F}(B)=c-A:B$, $B\in S_{d}$ is proper. Furthermore, for each $B_{1},B_{2} \in S_{d}$ satisfying $B_{1}\leq B_{2}$, we have
\[
\sup\left\{\frac{1}{2}B_{2}:E \st E\in S^{++}_{d},\; f(E)=1\right\}\geq \sup\left\{\frac{1}{2} B_{1}:E \st E\in S^{++}_{d},\; f(E)=1\right\}
\]
and hence $c-\mathcal{T}(\frac{1}{2}B_{2})\leq c-\mathcal{T}(\frac{1}{2}B_{1})$, where $\mathcal{T}:S_{d}\to \mathbb{R}$ is defined in (\ref{Trace by f}).  Thus, $\mathscr{F}(\cdot)=c-\mathcal{T}(\frac{1}{2}\cdot)$ is proper for each $c\in\mathbb{R}$.
\end{rem}
For the reader's convenience, we recall (the reader may consult \cite{Crandall-Ishii-Lions-1992}) the next 
\begin{defn}\label{def of subsolution}
	Let $\mathscr{F}:S_{d}\to \mathbb{R}$ be proper (see \eqref{proper}). A lower semicontinuous function $u: \mathbb{R}^{d}\to \mathbb{R}$ is a viscosity supersolution of $\mathscr{F}=0$ (a viscosity solution of $\mathscr{F}\geq 0$) in $\mathbb{R}^{d}$ provided that if $\varphi \in C^{2}(\mathbb{R}^{d})$ and $x \in \mathbb{R}^{d}$ is a local minimum of $u-\varphi$, then 
$
	\mathscr{F}(\nabla^{2} \varphi(x))\geq 0.
$
\end{defn} 
The definitions of a viscosity subsolution and  a solution of $\mathscr{F}=0$ are likewise, we refer to \cite{Crandall-Ishii-Lions-1992}. 
Notice that the function $\mathscr{F}$ in \eqref{proper} and in Definition~\ref{def of subsolution} can be discontinuous. Even more, allowing $\mathscr{F}$ to become infinite (see, for instance, \cite[Example~1.11]{Crandall-Ishii-Lions-1992}), we observe that $\mathscr{F}=-f^{*}:S_{d} \to \{0, -\infty\}$ is proper. In particular, since $-f^{*}\leq 0$, it follows that any upper semicontinuous function $\psi$ is a viscosity subsolution of $-f^{*}(\frac{1}{2}\nabla^{2}u)=0$. 
\begin{prop}\label{prop visc supersol} Let $\psi \in C_{b}(\mathbb{R}^{d})$ satisfy $-\psi=(-\psi)^{G}$ on $\mathbb{R}^{d}$. Then $\psi$ is a viscosity supersolution of $1-\mathcal{T}\left(\frac{1}{2}\nabla^{2}u\right)=0$ in $\mathbb{R}^{d}$. 
\end{prop}
\begin{proof}
 Let $x_{0}\in \mathbb{R}^{d}$, $\varphi \in C^{2}(\mathbb{R}^{d})$ and $x_{0}$ be a local minimum of $\psi-\varphi$. Then there exists $r>0$ such that $\psi(x)-\psi(x_{0})\geq \varphi(x)-\varphi(x_{0})$ for each $x\in \overline{B}_{r}(x_{0})$. Fix an arbitrary $A=(a_{ij})_{i,j=1}^{d}\in S^{++}_{d}$ such that $f(A)=1$. Let $g_{A}$ be the Green function of the elliptic operator $L=-\sum_{i,j=1}^{d}a_{ij}\partial_{ij}$ on $B_{r}(x_{0})$. Then, defining $\lambda=2Ag_{A}(\cdot,x_{0}) \mathcal{L}^{d}\mres B_{r}(x_{0}) \in \mathcal{M}(\mathbb{R}^{d}, S^{+}_{d})$ and $p=- A\nabla g_{A}(\cdot,x_{0}) \cdot \nu(\cdot) \mathcal{H}^{d-1}\mres \partial B_{r}(x_{0}) \in \mathcal{P}_{2}(\mathbb{R}^{d})$, where $\nu(x)$ is the outward pointing unit normal to $\partial B_{r}(x_{0})$ at $x$, we know that 
\begin{equation}\label{lcomp1.48}
\textup{tr}\left(\frac{1}{2}\nabla^{2}\lambda\right)=p-\delta_{x_{0}}
\end{equation}
(see the proof of Proposition~\ref{lem: 1.21}).  Since $\varphi-\varphi(x_{0})\leq \psi-\psi(x_{0})$ on $\overline{B}_{r}(x_{0})$,
\begin{equation}\label{est 1.49}
\int_{\mathbb{R}^{d}}\varphi\diff p - \varphi(x_{0})\le \int_{\mathbb{R}^{d}}\psi\diff p -\psi(x_{0})\leq G(x_{0},p),
\end{equation}
where the last estimate comes from the fact that $-\psi=(-\psi)^{G}$ (see Remark~\ref{rem 1.22}). On the other hand, in view of (\ref{lcomp1.48}), $\lambda$ is a competitor in the definition of $G(x_{0},p)=F(\delta_{x_{0}}, p)$ (see Proposition~\ref{prop: 1.7}). Using this, \eqref{assumpsublinearf} and the fact that $f(A)=1$, we obtain 
\[
G(x_{0},p)\leq \int_{\mathbb{R}^{d}}f\left(\frac{\diff \lambda}{\diff|\lambda|}\right)\diff |\lambda| = 2\int_{B_{r}(x_{0})}f(A)g_{A}(x,x_{0})\diff x=2\int_{B_{r}(x_{0})}g_{A}(x,x_{0})\diff x.
\]
This, together with (\ref{est 1.49}) and the fact that we can actually use $\varphi$ as a test function for (\ref{lcomp1.48}) (since the supports of the measures $\lambda$, $p$ and $\delta_{x_{0}}$ are contained in $\overline{B}_{r}(x_{0})$ and we can multiply $\varphi$ by a cutoff function equal to 1 on $B_{2r}(x_{0})$), implies that
\begin{equation}\label{est 1.50}
\int_{B_{r}(x_{0})} A:\nabla^{2}\varphi(x) g_{A}(x,x_{0})\diff x \leq 2\int_{B_{r}(x_{0})}g_{A}(x,x_{0})\diff x.
\end{equation}
Then $A:\frac{1}{2}\nabla ^{2}\varphi(x_{0})\leq 1$, because otherwise we could choose $r>0$ small enough such that for some $\varepsilon>0$ and for each $x \in B_{r}(x_{0})$ we would have $A:\frac{1}{2}\nabla ^{2}\varphi(x)- 1\geq \varepsilon$, which would lead to a contradiction with (\ref{est 1.50}), since $\int_{B_{r}(x_{0})}g_{A}(x,x_{0})\diff x>0$. Thus, $1-A:\frac{1}{2}\nabla^{2}\varphi(x_{0})\geq 0$ for each $A\in S^{++}_{d}$ satisfying $f(A)=1$. Therefore,
\[
1-\mathcal{T}\left(\frac{1}{2}\nabla^{2}\varphi(x_{0})\right)=1-\sup\left\{A:\frac{1}{2}\nabla^{2}\varphi(x_{0}) \st A\in S^{++}_{d},\; f(A)=1\right\}\geq 0.
\]
This, according to Definition~\ref{def of subsolution}, completes our proof of Proposition~\ref{prop visc supersol}.
\end{proof}
Next, we perform a smoothing procedure by convolution with a mollifier for a viscosity solution of $1-\mathcal{T}\left(\frac{1}{2}\nabla^{2}u\right)\geq 0$ in $\mathbb{R}^{d}$ to obtain the classical solution.
\begin{rem}\label{rem 1.30}
The same argument as in the proof of Proposition~\ref{prop visc supersol} implies that for each $c\in\mathbb{R}$ and for each lower semicontinuous function $w: \mathbb{R}^{d}\to \mathbb{R}$,  $w$ is a viscosity supersolution of $c-\mathcal{T}\left(\frac{1}{2}\nabla^{2}u\right)=0$ in $\mathbb{R}^{d}$ if and only if for each $A\in S^{++}_{d}$ such that $f(A)=1$, $w$ is a viscosity supersolution of  $c-A:\frac{1}{2}\nabla^{2}u=0$ in $\mathbb{R}^{d}$. 
\end{rem}
\begin{lemma}\label{lem 1.36}
	Let $A\in S^{++}_{d}$, $c\in \mathbb{R}$ and $w\in L^{1}_{\loc}(\mathbb{R}^{d})$ be lower semicontinuous on $\mathbb{R}^{d}$. Let $\eta \in C^{\infty}_{c}(\mathbb{R}^{d})$, $\supp(\eta)=\overline{B}_{1}(0)$, $\eta\geq 0$, $\eta(x)=\eta(-x)$, $\int_{\mathbb{R}^{d}}\eta\diff x=1$ and $\eta_{\varepsilon}(\cdot)=\varepsilon^{-d}\eta(\cdot/\varepsilon)$ for each  $\varepsilon>0$. Assume that $w$ is a viscosity supersolution of $c-\mathcal{T}\left(\frac{1}{2}\nabla^{2}u\right)=0$ in $\mathbb{R}^{d}$, where $\mathcal{T}$ is defined in (\ref{Trace by f}). Then for each $\varepsilon>0$, $c-\mathcal{T}\left(\frac{1}{2}\nabla^{2}w_{\varepsilon}\right) \geq 0$ in $\mathbb{R}^{d}$, where $w_{\varepsilon}(\cdot)=\int_{\mathbb{R}^{d}}w(y)\eta_{\varepsilon}(\cdot-y)\diff y$.
\end{lemma}
\begin{proof}
Let $A\in S^{++}_{d}$ be such that $f(A)=1$. Let $B\in S^{++}_{d}$ be the unique matrix such that $A=B^{-2}$ and for each $x \in \mathbb{R}^{d}$, define $v(x)=w(B^{-1}x)$ and $h(x)=v(x)-\frac{c}{d}|x|^{2}$. \\
\noindent \emph{Step 1.} We prove that the following assertions are equivalent. 
\begin{enumerate}[label=(\roman*)]
\item \label{item1supersolequiv}
$w$ is a viscosity supersolution of $c-A:\frac{1}{2}\nabla^{2}u= 0$ in 
$\mathbb{R}^{d}$.
\item \label{item2supersolequiv}
$v$ is a viscosity supersolution of $c-\frac{1}{2}\Delta u= 0$ in $\mathbb{R}^{d}$.
\item \label{item3supersolequiv}
$h$ is a viscosity supersolution of $-\frac{1}{2}\Delta u= 0$ in $\mathbb{R}^{d}$.
\end{enumerate}
The equivalence \ref{item1supersolequiv}$\Leftrightarrow$ \ref{item2supersolequiv} directly follows from $A:\frac{1}{2}\nabla^{2}\psi(x)= \frac{1}{2} \Delta \varphi(Bx)$ whenever $\psi(x)=\varphi(Bx)$ for each $x\in \mathbb{R}^{d}$ and $\varphi \in C^{2}(\mathbb{R}^{d})$, while the equivalence \ref{item2supersolequiv}$\Leftrightarrow$ \ref{item3supersolequiv} is straightforward. \\
\noindent \emph{Step 2.} We prove that the inequality $c-\frac{1}{2}\Delta v_{\varepsilon}\geq 0$ holds in $\mathbb{R}^{d}$, where $v_{\varepsilon}(\cdot)=\int_{\mathbb{R}^{d}}v(y)\widetilde{\eta}_{\varepsilon}(\cdot-y)\diff y$ and $\widetilde{\eta}_{\varepsilon}(\cdot)=(\det(B))^{-1}\eta_{\varepsilon}(B^{-1}\cdot)$. By definition,  $\widetilde{\eta}_{\varepsilon}\in C^{\infty}_{c}(\mathbb{R}^{d})$, $\supp(\widetilde{\eta}_{\varepsilon})\subset \{Bx \st x \in \overline{B}_{\varepsilon}(0)\}$, $\widetilde{\eta}_{\varepsilon}\geq 0$, $\widetilde{\eta}_{\varepsilon}(x)=\widetilde{\eta}_{\varepsilon}(-x)$ and $\int_{\mathbb{R}^{d}}\widetilde{\eta}_{\varepsilon}(x)\diff x =1$. Since $w$ is a viscosity supersolution of $c-A:\frac{1}{2}\nabla^{2}u = 0$ in $\mathbb{R}^{d}$ (see Remark~\ref{rem 1.30}), by the equivalence \ref{item1supersolequiv}$\Leftrightarrow$ \ref{item3supersolequiv} proved in Step~1, $h$ is a viscosity supersolution of $-\frac{1}{2}\Delta u=0$ in $\mathbb{R}^{d}$. Let us show that $h_{\varepsilon}(\cdot)=\int_{\mathbb{R}^{d}}h(y)\widetilde{\eta}_{\varepsilon}(\cdot-y)\diff y$ is superharmonic in $\mathbb{R}^{d}$, which is equivalent to the property $c-\frac{1}{2}\Delta v_{\varepsilon}\geq 0$ in $\mathbb{R}^{d}$. The proof is based on the fact that $h$ is viscosity superharmonic in $\mathbb{R}^{d}$ if and only if $\fint_{\partial B_{r}(x_{0})}h(x) \diff \mathcal{H}^{d-1}(x) \leq h(x)$ for each $x_{0} \in \mathbb{R}^{d}$ and $r>0$, namely $h$ is superharmonic in $\mathbb{R}^{d}$ (see Definition~\ref{defofsubharmonic}). Thus, changing the variables and applying Fubini's theorem, we obtain
	\begin{align*}
	\fint_{\partial B_{r}(x)}h_{\varepsilon}(y)\diff \mathcal{H}^{d-1}(y)&=\fint_{\partial B_{r}(x)}\diff \mathcal{H}^{d-1}(y)\int_{\mathbb{R}^{d}}h(z)\widetilde{\eta}_{\varepsilon}(y-z)\diff z\\	
	&=\int_{\mathbb{R}^{d}}\diff z\fint_{\partial B_{r}(x)}h(y-z)\widetilde{\eta}_{\varepsilon}(z)\diff \mathcal{H}^{d-1}(y)\\
	&\leq \int_{\mathbb{R}^{d}}h(x-z)\widetilde{\eta}_{\varepsilon}(z)\diff z=h_{\varepsilon}(x),
	\end{align*}
	where we have also used that $\fint_{\partial B_{r}(x-z)}h(y)\diff \mathcal{H}^{d-1}(y)\leq h(x-z)$, since $h$ is superharmonic in $\mathbb{R}^{d}$. Then, by Definition~\ref{defofsubharmonic}, $h_{\varepsilon}$ is superharmonic in $\mathbb{R}^{d}$. This, since $h_{\varepsilon}\in C^{2}(\mathbb{R}^{d})$, implies that $-\frac{1}{2}\Delta h_{\varepsilon}\geq 0$ in $\mathbb{R}^{d}$ and hence $c-\frac{1}{2}\Delta v_{\varepsilon}\geq 0$ in $\mathbb{R}^{d}$ as desired.\\
\noindent \emph{Step 3.} Using the result of Step~2, the fact that $A:\nabla^{2}\eta_{\varepsilon}(x)=\det(B)\Delta\widetilde{\eta}_{\varepsilon}(Bx)$ for each $x\in \mathbb{R}^{d}$ and changing the variables, for each $x \in \mathbb{R}^{d}$ we deduce the following
	\begin{align*}
	0\leq c- \frac{1}{2}\Delta v_{\varepsilon}(Bx)&=c-\frac{1}{2}\int_{\mathbb{R}^{d}}v(y)\Delta\widetilde{\eta}_{\varepsilon}(Bx-y)\diff y\\ &=c-\frac{(\det(B))^{-1}}{2}\int_{\mathbb{R}^{d}}v(y)A:\nabla^{2}\eta_{\varepsilon}(x-B^{-1}y)\diff y\\
	&=c-\frac{1}{2}\int_{\mathbb{R}^{d}}v(By)A:\nabla^{2}\eta_{\varepsilon}(x-y)\diff y=c-A:\frac{1}{2}\nabla^{2}w_{\varepsilon}(x).
	\end{align*}
	This, in view of Remark~\ref{rem 1.30} and the fact that $A\in S^{++}_{d}$ satisfying $f(A)=1$ was arbitrarily chosen (observe that a \emph{classical} supersolution is a viscosity supersolution, and a viscosity supersolution of class $C^{2}$ is a \emph{classical} supersolution), completes our proof of Lemma~\ref{lem 1.36}.
\end{proof}
Now we characterize the dual competitors in \eqref{thmdualA3A4} as viscosity solutions of  $-f^{*}(\frac{1}{2}\nabla^{2}u) = 0$ in $\mathbb{R}^{d}$.
\begin{prop}\label{prop G-ivariant=VS}
Let $\psi \in C_{b}(\mathbb{R}^{d})$ and \eqref{assumpcoercivef} hold. Then $-\psi=(-\psi)^{G}$ on $\mathbb{R}^{d}$ if and only if $\psi$ is a viscosity solution of $-f^{*}(\frac{1}{2}\nabla^{2}u) = 0$ in $\mathbb{R}^{d}$.
\end{prop}
\begin{rem}
In \cite[Theorem~4.2]{Tan-Touzi-2013}, in the Markovian case, it was proved that the dynamic value function is a viscosity solution of the Hamilton-Jacobi-Bellman equation. It is worth noting that, unlike \cite{Tan-Touzi-2013}, in our case $f^{*}$ is not continuous and $\dom(f^{*})$ is a closed convex subset of $S_{d}$. Furthermore, we provide an analytical proof of Proposition~\ref{prop G-ivariant=VS} that is different from the proof of \cite[Theorem~4.2]{Tan-Touzi-2013}, showing in addition that if $\psi$ is a viscosity solution of the Hamilton-Jacobi-Bellman equation $-f^{*}(\frac{1}{2}\nabla^{2}u)=0$ in $\mathbb{R}^{d}$, then $-\psi$ is invariant under the $G$-transform.
\end{rem}
\begin{proof}
By Proposition~\ref{prop visc supersol}, \eqref{Tf1} and Definition~\ref{def of subsolution} (where we allow $\mathscr{F}$ to be discontinuous and, even more, to become infinite), if $\psi \in C_{b}(\mathbb{R}^{d})$ and $-\psi=(-\psi)^{G}$ on $\mathbb{R}^{d}$, then $\psi$ is a viscosity solution of $-f^{*}(\frac{1}{2}\nabla^{2} u)=0$ in $\mathbb{R}^{d}$. 

Conversely, if $\psi \in C_{b}(\mathbb{R}^{d})$ is a viscosity solution of $-f^{*}(\frac{1}{2}\nabla^{2}u)=0$ in $\mathbb{R}^{d}$, by \eqref{Tf1} and Definition~\ref{def of subsolution}, $\psi$ is a viscosity supersolution of $1-\mathcal{T}(\frac{1}{2}\nabla^{2}u)=0$ in $\mathbb{R}^{d}$. Next, applying Lemma~\ref{lem 1.36}, we deduce that for each $\varepsilon>0$, it holds $1-\mathcal{T}(\frac{1}{2} \nabla^{2}\psi_{\varepsilon})\geq 0$ in $\mathbb{R}^{d}$, where $\psi_{\varepsilon}(\cdot)=\int_{\mathbb{R}^{d}}\psi(y)\eta_{\varepsilon}(\cdot-y)\diff y$ and  $\eta$ is the mollifier of Lemma~\ref{lem 1.36}. Hence $-f^{*}(\frac{1}{2}\nabla^{2}\psi_{\varepsilon}) = 0$ in $\mathbb{R}^{d}$ (see \eqref{Tf1}). This, according to Proposition~\ref{lem: 1.21}, yields the equality $-\psi_{\varepsilon}=(-\psi_{\varepsilon})^{G}$ on $\mathbb{R}^{d}$ for each $\varepsilon>0$. According to Proposition~\ref{prop lsc for G}, for each $x\in \mathbb{R}^{d}$ and $\varepsilon>0$, there exist $p, p_{\varepsilon} \in \mathcal{P}_{2}(\mathbb{R}^{d})$ such that
\begin{equation}\label{exmincoerciveGep}
(-\psi_{\varepsilon})^{G}(x)=\int_{\mathbb{R}^{d}}-\psi_{\varepsilon} \diff p_{\varepsilon} + G(x,p_{\varepsilon})
\end{equation}
and
\begin{equation}\label{exmincoerciveG}
(-\psi)^{G}(x)=\int_{\mathbb{R}^{d}} -\psi \diff p + G(x,p).
\end{equation}
Using \eqref{exmincoerciveG}, the fact that $\psi\in C_{b}(\mathbb{R}^{d})$, Lebesgue's dominated convergence theorem, \eqref{exmincoerciveGep} and the equality $-\psi_{\varepsilon}=(-\psi_{\varepsilon})^{G}$, for each $x\in \mathbb{R}^{d}$, we have
\begin{align*}
(-\psi)^{G}(x)&=\int_{\mathbb{R}^{d}}-\psi \diff p + G(x,p)\\ &= \lim_{\varepsilon \to 0+}\int_{\mathbb{R}^{d}}-\psi_{\varepsilon} \diff p + G(x,p)\\
& \ge \lim_{\varepsilon \to 0+}\int_{\mathbb{R}^{d}}-\psi_{\varepsilon} \diff p_{\varepsilon} + G(x,p_{\varepsilon})\\
&= \lim_{\varepsilon \to 0+}-\psi_{\varepsilon}(x)\\
&=-\psi(x).
\end{align*}
This, in view of the fact that $-\psi\geq (-\psi)^{G}$ on $\mathbb{R}^{d}$, proves the equality $-\psi=(-\psi)^{G}$ on $\mathbb{R}^{d}$ and completes our proof of Proposition~\ref{prop G-ivariant=VS}.
\end{proof}
The next corollary is a direct consequence of Theorem~\ref{thmdualHA1-A4} and Proposition~\ref{prop G-ivariant=VS}.
\begin{cor}\label{corHthroughVS}
Let $\mu, \nu \in \mathcal{P}_{2}(\mathbb{R}^{d})$ and \eqref{assumpcoercivef}, \eqref{assumpgrowthf} hold. Then
\[
H(\mu,\nu)=\sup\left\{\langle \nu -\mu, \psi \rangle \st \psi \in C_{b}(\mathbb{R}^{d})\; \text{is a viscosity solution of}\; -f^{*}\left(\frac{1}{2} \nabla^{2} u\right)=0\,\ \text{in}\,\ \mathbb{R}^{d}\right\}.
\]
\end{cor}

\subsection{$F=H$ under the assumptions \eqref{assumpcoercivef}, \eqref{assumpgrowthf}}

\begin{theorem}\label{th F=G}
	Let \eqref{assumpcoercivef} and \eqref{assumpgrowthf} hold. Then $F=H$ on $\mathcal{P}_{2}(\mathbb{R}^{d})\times \mathcal{P}_{2}(\mathbb{R}^{d})$.
\end{theorem}
\begin{proof}
	Let $\psi \in C_{b}(\mathbb{R}^{d})$ be a viscosity solution of $-f^{*}(\frac{1}{2}\nabla^{2}u)=0$ in $\mathbb{R}^{d}$. Then, performing the smoothing procedure as in the proof of Proposition~\ref{prop G-ivariant=VS}, we obtain a sequence of functions $\psi_{n} \in C^{2}_{b}(\mathbb{R}^{d})$ such that $-f^{*}(\frac{1}{2}\nabla^{2}\psi_{n})=0$ in $\mathbb{R}^{d}$, $\int_{\mathbb{R}^{d}}\psi_{n} \diff \mu \to \int_{\mathbb{R}^{d}} \psi \diff \mu$ and $\int_{\mathbb{R}^{d}} \psi_{n} \diff \nu \to \int_{\mathbb{R}^{d}} \psi \diff \nu$ as $n \to +\infty$. Using this, Corollary~\ref{corHthroughVS} and Theorem~\ref{prop: 1.13}, we obtain $H(\mu,\nu)\le F(\mu, \nu)$. Therefore, $H(\mu,\nu)=F(\mu,\nu)$, since $H(\mu,\nu)\geq F(\mu,\nu)$ by Proposition~\ref{prop lower bound}. This completes our proof of Theorem~\ref{th F=G}.
	\end{proof}
\noindent \textbf{Example}~\ref{tracefirstexample}.
	For each $A\in S_{d}$, 
	\begin{equation}\label{exoffbytrace}
	f(A)=\begin{cases}
	\tr(A)\,\ &\text{if}\,\  A\in S^{+}_{d},\\
	+\infty \,\ &\text{otherwise}.
	\end{cases}
	\end{equation}
	Clearly, $f$ satisfies \eqref{assumpdom}-\eqref{assumpgrowthf}. Then for each $A\in S_{d}$, 
	\begin{equation}\label{symAtr}
	\mathcal{T}(A)\leq 1 \Leftrightarrow A - I_{d}\leq 0. 
	\end{equation}
	Indeed, $A:M\leq 1$ for each $M\in S^{++}_{d}$ such that $\tr(M)=1$ if and only if $(A-I_{d}):M\leq 0$ for each $M\in S^{+}_{d}$,  which is equivalent to say that $A-I_{d}\le 0$. Using this and \eqref{Trace by f}, one deduces \eqref{symAtr}.
		
	Given $\mu, \nu \in \mathcal{P}_{2}(\mathbb{R}^{d})$, applying Theorem~\ref{prop: 1.13}, \eqref{Tf1} and \eqref{symAtr}, we compute
	\begin{align*}
	F(\mu, \nu)&=\sup\left\{\langle \nu-\mu, \psi\rangle \st \psi \in C^{2}_{b}(\mathbb{R}^{d}),\; \frac{1}{2}\nabla^{2}\psi(x)-I_{d}\leq 0 \,\ \forall x\in \mathbb{R}^{d}\right\}\\
	&=\left\langle \nu-\mu, |\cdot|^{2} \right\rangle + \sup\{\langle \nu - \mu, u\rangle \st u\in C^{2}_{b}(\mathbb{R}^{d}),\; u \,\ \text{is concave on}\,\ \mathbb{R}^{d}\}\\
	&=\begin{cases}
    \var(\nu)-\var(\mu) \,\ &\text{if}\,\  \mu\le_{c}\nu, \\ \numberthis \label{expressionftrace1}
    +\infty \,\ &\text{otherwise}    
    \end{cases}\\
    &=\begin{cases}
\displaystyle f\left(\int_{\mathbb{R}^{d}} x\otimes x \diff (\nu(x)- \mu(x))\right) &\text{if}\,\ \mu\le_{c} \nu, \\ \numberthis \label{expressionftrace2}
+\infty \,\ &\text{otherwise},
\end{cases}
    \end{align*}
where we have used that $\langle \nu- \mu, |\cdot|^{2}\rangle =\var(\nu)-\var(\mu)$ if $\mu\le_{c} \nu$ (in this case $[\mu]=[\nu]$). For each $a \in \mathbb{R}$ and $\xi \in \mathbb{R}^{d}$, the function  $\psi(x)=a+\xi\cdot x + |x|^{2}$ is a dual optimizer for $F(\mu,\nu)$, which is understood in the setting of \eqref{Dual}, when $F(\mu,\nu)<+\infty$ and  $f$ is defined by \eqref{exoffbytrace}.  By Theorem~\ref{th F=G}, $F=H$ when $f$ is defined by \eqref{exoffbytrace}.

\noindent \textbf{Example}~\ref{levsecondexample}.
	For each $A\in S_{d}$, 
	\begin{equation}\label{flmaxexample}
	f(A)=\begin{cases}
	\lambda_{\max}(A)\,\ &\text{if}\,\  A\in S^{+}_{d},\\
	+\infty \,\ &\text{otherwise},
	\end{cases}
	\end{equation}
	where $\lambda_{\max}(A)=\max\{Ax\cdot x\st |x|=1\}$ is the largest eigenvalue of $A$. Clearly, $f$ satisfies \eqref{assumpdom}-\eqref{assumpgrowthf}. Every matrix $A\in S_{d}$ has $d$ real eigenvalues $\lambda_{1}, \dotsc, \lambda_{d}$ such that $|\lambda_{1}|\geq |\lambda_{2}|\geq \dotsc \geq |\lambda_{d}|$ with corresponding eigenvectors $v_{1},\dotsc, v_{d} \in \mathbb{R}^{d}$ forming an orthonormal basis of $\mathbb{R}^{d}$. Denote $J_{+}=\{j\in\{1,\dotsc,d\} \st \lambda_{j}\geq 0\}$, $J_{-}=\{1,\dotsc,d\}\backslash J_{+}$, $A^{+}=\sum_{j \in J_{+}}\lambda_{j}v_{j}\otimes v_{j}$, $A^{-}=\sum_{j\in J_{-}}\lambda_{j}v_{j}\otimes v_{j}$ so that $A=A^{+}+A^{-}$.  We claim that 
	\begin{equation}\label{lev}
	\mathcal{T}(A)\leq 1 \Leftrightarrow \sum_{j\in J_{+}}\lambda_{j}\leq 1. 
	\end{equation}
	Indeed, if $A \in S_{d}$, then $A=PDP^{\mathrm{T}}$, where  $D=\textup{diag}(\lambda_{1},\dotsc,\lambda_{d})$ and $P=(v_{1}\dotsc v_{d})\in S_{d}$ is the orthonormal matrix with columns $v_{1},\dotsc, v_{d}$. Let $\varepsilon \in (0,1)$ be fairly small. Define $M=PEP^{\mathrm{T}}$, where $E=(e_{ij})_{i,j=1}^{d} \in S^{++}_{d}$ is the diagonal matrix satisfying $e_{ii}=\varepsilon$ if $\lambda_{i}<0$ and $e_{ii}=1$ otherwise. Next, we compute
    \begin{align*}
    A:M=\tr(PDP^{\mathrm{T}}PEP^{\mathrm{T}})=\tr(DE)=\sum_{j\in J_{+}}\lambda_{j} + \varepsilon \sum_{j\in J_{-}} \lambda_{j}.
    \end{align*}
   Thus, assuming that $\mathcal{T}(A)\leq 1$ and letting $\varepsilon \to 0+$, we have $\sum_{j\in J_{+}}\lambda_{j}\leq 1$. Conversely, suppose that $\sum_{j\in J_{+}}\lambda_{j}\leq 1$. For each $M\in S^{++}_{d}$ such that $\lambda_{\max}(M)=1$,
   \begin{align*}
   A:M=(A^{+}+A^{-}):M\leq\tr(A^{+}M)&=\sum_{j\in J_{+}}\lambda_{j}Mv_{j}\cdot v_{j}\leq \sum_{j\in J_{+}}\lambda_{j}\lambda_{\max}(M)\leq 1.
   \end{align*}
   This proves our claim.

   Given $\mu, \nu \in \mathcal{P}_{2}(\mathbb{R}^{d})$, using Theorem~\ref{prop: 1.13}, \eqref{Tf1} and \eqref{lev}, we deduce that
   \begin{equation}\label{Fwhenf=largesteigenvalue}
   F(\mu,\nu)=
   \sup\left\{\langle \nu-\mu, \psi \rangle \st \psi \in C^{2}_{b}(\mathbb{R}^{d}),\; \sum_{j\in J_{+}}\lambda_{j}\left(\frac{1}{2}\nabla^{2}\psi(x)\right)\leq 1 \,\ \forall x \in \mathbb{R}^{d}\right\}.
   \end{equation}
Since $\frac{1}{d}\tr(A)\leq \lambda_{\max}(A)\leq \tr(A)$ for each $A\in S^{+}_{d}$, we have $\frac{1}{d}\widetilde{F}(\mu,\nu)\leq F(\mu,\nu) \leq \widetilde{F}(\mu,\nu)$, where 
\[
\widetilde{F}(\mu,\nu)=\inf\left\{\int_{\mathbb{R}^{d}}\tr\left(\frac{\diff \lambda}{\diff |\lambda|}\right)\diff|\lambda| \, | \, \textup{tr}\left(\frac{1}{2} \nabla^{2}\lambda\right)=\nu-\mu\right\}. 
\]
In particular, if $d=1$, then $\widetilde{F}(\mu, \nu) =F(\mu, \nu)$ and the expression of $F(\mu, \nu)$ in terms of $\mu$ and $\nu$ comes from \eqref{expressionftrace1}, \eqref{expressionftrace2}. 
Then, in view of \eqref{expressionftrace1}, 
\begin{equation*}
F(\mu,\nu)<+\infty \Leftrightarrow \mu\leq_{c}\nu.
\end{equation*}
Assume that $\mu \le_{c} \nu$. Define 
\[
\alpha=\sup\{\langle \nu -\mu, \psi \rangle \st \psi \in C^{2}_{b}(\mathbb{R}^{d})\,\ \text{is convex and}\,\ \Delta \psi=2 \,\ \text{in} \,\ \mathbb{R}^{d}\}.
\] 
For each $\xi \in \partial B_{1}(0)$, $x \in \mathbb{R}^{d} \mapsto |x\cdot \xi|^{2}$ belongs to $C^{2}_{b}(\mathbb{R}^{d})$, is convex and $\Delta |x\cdot \xi|^{2}=2$ in $\mathbb{R}^{d}$. Then
   \begin{equation}\label{estfbyalpha}
   \begin{split}
   f\left(\int_{\mathbb{R}^{d}}x\otimes x \diff (\nu(x)-\mu(x))\right) &=\max_{\xi \in \partial B_{1}(0)}\xi \otimes \xi : \int_{\mathbb{R}^{d}} x\otimes x \diff(\nu(x)-\mu(x)) \\
   &=\max_{\xi \in \partial B_{1}(0)} \int_{\mathbb{R}^{d}} |x\cdot \xi|^{2} \diff(\nu(x)-\mu(x))\le \alpha. 
   \end{split}
   \end{equation}
Let $\psi \in C^{2}_{b}(\mathbb{R}^{d})$ be a convex function such that $\Delta\psi=2$ in $\mathbb{R}^{d}$.  Define $\varphi=\psi - \frac{1}{d}|\cdot|^{2} \in C^{2}_{b}(\mathbb{R}^{d})$. Then $\varphi$ is harmonic in $\mathbb{R}^{d}$. Since $\varphi \in \Phi_{2}(\mathbb{R}^{d})$, using \cite[Theorem~2.10]{Gilbarg-Trudinger-1998}, we deduce that there exists $A \in S_{d}$ such that $\tr(A)=0$, $A+\frac{I_{d} }{d} \geq 0$ and $\nabla^{2} \varphi(x)= 2A$ for each $x \in \mathbb{R}^{d}$. Then there exist $a \in \mathbb{R}$ and $\zeta \in \mathbb{R}^{d}$ such that $\psi(x)=a+\zeta\cdot x + (A+\frac{I_{d}}{d}):x \otimes x$ for each $x \in \mathbb{R}^{d}$. Thus, taking into account that $\mu \leq_{c} \nu$, we have
\begin{equation}\label{alphanewmatrix}
\alpha=\sup\left\{\int_{\mathbb{R}^{d}}\left(A+\frac{I_{d}}{d}\right):x \otimes x \diff (\nu (x)-\mu(x))\st A+\frac{I_{d}}{d} \in S^{+}_{d},\; \tr(A)=0\right\}.
\end{equation}
For each $A+\frac{I_{d}}{d} \in S^{+}_{d}$, where $\tr(A)=0$, there exist eigenvalues $\lambda_{1} \geq \dotsc >\lambda_{d} \geq 0$ and corresponding eigenvectors $v_{1}, \dotsc, v_{d}$ forming an orthonormal basis of $\mathbb{R}^{d}$ such that we have $A+\frac{I_{d}}{d}=\sum_{i=1}^{d} \lambda_{i} v_{i}\otimes v_{i}$ and $\sum_{i=1}^{d}\lambda_{i}=1$. This, together with \eqref{Fwhenf=largesteigenvalue}, \eqref{estfbyalpha} and \eqref{alphanewmatrix}, yields
\begin{equation*}
f\left(\int_{\mathbb{R}^{d}} x \otimes x \diff (\nu(x)-\mu(x)) \right) = \alpha \le F(\mu, \nu).
\end{equation*}
Next, let $\psi \in C^{2}_{b}(\mathbb{R}^{d})$ be such that $\sum_{j \in J_{+}}\lambda_{j}(\frac{1}{2}\nabla^{2}\psi(x)) \leq 1$ for each $x \in \mathbb{R}^{d}$. Then there exists a subharmonic function $u \in C^{2}_{b}(\mathbb{R}^{d})$ such that $\Delta u = 2-\Delta \psi \geq0$. Using \cite[Theorem~2.10]{Gilbarg-Trudinger-1998}, we obtain a matrix $A \in S_{d}$ such that $\tr(A)=0$ and $\nabla^{2}u(x)+\nabla^{2}\psi(x)=2A +\frac{2I_{d}}{d}$ for each $x \in \mathbb{R}^{d}$. Then there exist $a \in \mathbb{R}$ and $\zeta \in\mathbb{R}^{d}$ such that $\psi(x)=a+\zeta \cdot x + (A+\frac{I_{d}}{d}): x \otimes x - u(x)$ for each $x \in \mathbb{R}^{d}$. Taking into account the constraint in \eqref{Fwhenf=largesteigenvalue} and the assumption $\mu\le_{c} \nu$, we obtain
\begin{equation*}
\begin{split} 
F(\mu,\nu)=\sup\Biggl\{\int_{\mathbb{R}^{d}} \Bigl(\Bigl(A+\frac{I_{d}}{d}\Bigr): x\otimes x &- u(x) \Bigr) \diff (\nu(x)-\mu(x))\st A \in S_{d},\; \tr(A)=0,\; u\in C^{2}_{b}(\mathbb{R}^{d}), \\ &\Delta u \geq 0 \,\ \text{in}\,\ \mathbb{R}^{d}\,\ \text{and}\,\ \sum_{j \in J_{+}} \lambda_{j}\Bigl(\Bigl(A+\frac{I_{d}}{d}\Bigr)-\frac{1}{2} \nabla^{2} u(x)\Bigr) \leq 1 \,\ \forall x \in \mathbb{R}^{d} \Biggr\}.
\end{split}
\end{equation*}
\noindent \textbf{Conjecture.} If $\mu \le_{c} \nu$, then
\[
f\left(\int_{\mathbb{R}^{d}} x \otimes x \diff (\nu(x)-\mu(x)) \right) = F(\mu, \nu).
\]
As in the example~\ref{tracefirstexample}, in our opinion, a dual optimizer for $F(\mu, \nu)$ should be sought among convex functions when $f$ is defined by \eqref{flmaxexample} and $\mu \le_{c} \nu$. We expect that such an optimizer exists and that its Laplacian is equal to 2 in $\mathbb{R}^{d}$. In dimension 1 and in general when $\mu\le_{sh} \nu$, the above conjecture is true and
\begin{align*}
F(\mu, \nu) & = \frac{1}{d} \var(\nu) - \frac{1}{d} \var(\mu),
\end{align*}
where for each $a \in \mathbb{R}$, for each $\zeta \in \mathbb{R}^{d}$ and for each $A\in S_{d}$ such that $A+\frac{I_{d}}{d}\geq 0$ and $\tr(A)=0$, the function $\psi(x)=a+\zeta \cdot x + (A+\frac{I_{d}}{d}):x \otimes x$ is a dual optimizer for $F(\mu, \nu)$, which is understood in the setting of \eqref{Dual}. 

By Theorem~\ref{th F=G}, $F=H$ when $f$ is defined by \eqref{flmaxexample}. 
   
\section{Acknowledgments}
I would like to warmly thank Guy Bouchitté for introducing me to this problem, for stimulating discussions, valuable comments on the manuscript and also for bringing \cite{Huesmann-Trevisan-2019} to my attention. Some discussions with Thierry Champion and Jean-Jacques Alibert were useful. I am also grateful to the anonymous referee for carefully reading the paper and for comments and suggestions that helped improve the paper. This work was partially supported by the Projet de Recherche T.0229.21 ``Singular Harmonic Maps and Asymptotics of Ginzburg-Landau Relaxations" of the Fonds de la Recherche Scientifique--FNRS. 
\titleformat{\section}{\normalfont\Large\bfseries}{\appendixname~\thesection.}{0.7em}{}
\begin{small}
\bibliography{bib01}
\bibliographystyle{plain}
\end{small}
\end{document}